\documentclass[12pt]{article}
\usepackage{graphicx}
\usepackage[ruled,vlined]{algorithm2e}
\usepackage{appendix}
\usepackage{amsmath}
\usepackage{amssymb}
\usepackage{amsthm}
\usepackage{multirow}
\usepackage{color}
\usepackage{longtable}
\usepackage{array}
\usepackage{url}
\usepackage{booktabs}
\usepackage{float}
\usepackage{mathtools}
\usepackage{tikz}
\usepackage{caption}
\usepackage{bm}

\newtheorem{theorem}{Theorem}[section]
\newtheorem{definition}[theorem]{Definition}
\newtheorem{lemma}[theorem]{Lemma}
\newtheorem{example}[theorem]{Example}
\newtheorem{construction}[theorem]{Construction}

\newtheorem{proposition}[theorem]{Proposition}
\newtheorem{problem}[theorem]{Problem}

\newtheorem{claim}[theorem]{Claim}
\theoremstyle{definition}
\newtheorem{remark}[theorem]{Remark}
\newtheorem{program}{Program}

\title{A polynomial resultant approach to algebraic constructions of extremal graphs}
\author{Tao Zhang$^{\text{a}}$,~ Zixiang Xu$^{\text{b}}$~ and  Gennian Ge$^{\text{b}}$\\
\footnotesize $^{\text{a}}$ Institute of Mathematics and Interdisciplinary Sciences, Xidian University, Xi'an 710071, China.\\
\footnotesize $^{\text{b}}$ School of Mathematical Sciences, Capital Normal University, Beijing 100048, China.\\
}

\begin{document}
\date{}

\maketitle

\begin{abstract}
The Tur\'{a}n problem asks for the largest number of edges ex$(n,H)$ in an $n$-vertex graph not containing a
fixed forbidden subgraph $H$, which is one of the most important problems in extremal graph theory. However the order of
magnitude of ex$(n,H)$ for bipartite graphs is known only in a handful of cases. In particular, giving explicit constructions of extremal graphs is very challenging in this field. In this paper, we develop a polynomail resultant approach to algebraic construction of explicit extremal graphs, which can efficiently decide whether a specified structure exists. A key insight in our approach is the multipolynomial resultant, which is a fundamental tool of computational algebraic geometry. Our main results include the matched lowers bounds for Tur\'{a}n number of $1$-subdivision of $K_{3,t_{1}}$ and linear Tur\'{a}n number of Berge theta hyerpgraph $\Theta_{3,t_{2}}^{B}$ with $t_{1}=25$ and $t_{2}=217$. Moreover, the constant $t_{1}$ improves the random algebraic construction of Bukh and Conlon~[Rational exponents in extremal graph theory, J. Eur. Math. Soc. 20 (2018), 1747-1757] and makes progress on the known estimation for the smallest value of $t_{1}$ concerning a problem posed by Conlon, Janzer and Lee ~[More on the extremal number of subdivisions, Combinatorica, to appear], while the constant $t_{2}$ improves a result of He and Tait~[Hypergraphs with few berge paths of fixed length between vertices, SIAM J. Discrete Math., 33(3), 1472-1481].
\medskip


\smallskip

\end{abstract}

\section{Introduction}
Given a graph $H$, the Tur\'{a}n number $\text{ex}(n,H)$ is the maximum number of edges in an $n$-vertex graph that does not contain $H$ as a subgraph. The estimation of $\text{ex}(n,H)$ for various graphs $H$ is one of the most important problems in extremal graph theory. For general graph $H,$ the famous Erd\H{o}s-Stone-Simonovits Theorem \cite{1966ESS, 1946ErodsBAMS} gives that
\begin{align*}
\text{ex}(n,H)=\bigg(1-\frac{1}{\chi(H)-1}+o(1)\bigg)\binom{n}{2},
\end{align*}
where $\chi(H)$ is the chromatic number of $H$. This theorem asymptotically solves the problem when $\chi(H)\ge3$. However, finding good bounds for $\text{ex}(n,H)$ for bipartite graphs $H$ is in general very difficult. The order of magnitude of $\text{ex}(n,H)$ is not known even in some very basic cases such as when $H$ is the even cycle $C_{8}$, complete bipartite graph $K_{4,4}$ or the $3$-dimensional cube $Q_{3}$. In general, complete bipartite graphs and even cycles are two of the most important classes of bipartite graphs.
A well-known theorem of K\"{o}v\'{a}ri, S\'{o}s and Tur\'{a}n~\cite{Kovari1954} showed that $\text{ex}(n,K_{s,t})=O(n^{2-1/s})$ for any integers $t\ge s$. In 1974, Bondy and Simonovits~\cite{BondyEvenCycle1974} gave a general upper bound $\text{ex}(n,C_{2\ell})=O(n^{1+\frac{1}{\ell}})$.

Showing matched lower bounds for Tur\'{a}n number of a specified bipartite graph is one of the most challenging problems in extremal graph theory. Usually there are three types of constructions as follows:
\begin{enumerate}
  \item \textbf{Probabilistic method with alterations}, which is quite general and easy to apply, but usually
does not give tight bounds. For example, for a given bipartite graph $H$ with $a$ vertices and $b$ edges, the standard probabilistic argument will give $\text{ex}(n,H)=\Omega(n^{\frac{2b-a}{b-1}}).$ In particular, when $H=K_{s,t}$ with $s\le t$, this method gives $\text{ex}(n,K_{s,t})=\Omega(n^{2-\frac{s+t-2}{st-1}})$, which does not match the upper bound $\text{ex}(n,K_{s,t})=O(n^{2-\frac{1}{s}})$.
  \item \textbf{Explicit constructions}, which often give tight bounds but appear to be somewhat magical and only work
in certain special situations. For example, one can view the vertices as the points in vector space over finite field, and two vertices are adjacent if they satisfy a system of equations in finite field. Using the above idea, when $s=2,3,$ matching lower bounds for $\text{ex}(n,K_{s,t})$ were found in \cite{Brown1966, Erdos1966}. For general values of $s$, Alon Koll\'{a}r, R\'{o}nyai and Szab\'{o}~\cite{ARS99, KRS96} constructed the so-called projective norm graphs, which yields $\text{ex}(n,K_{s,t})=\Omega(n^{2-\frac{1}{s}})$ when $t\geqslant (s-1)!+1$. Moreover, using the relation with finite geometry, Benson~\cite{1966Benson} and Singleton~\cite{1996Singleton} constructed extremal graphs for $C_{6}$ and $C_{10}$ in 1966.  Several alternative constructions can be found in~\cite{2021Conlon, 1995Lazebnik, 2005Mellinger, 1991Wenger}.
  \item \textbf{Random algebraic method}, which combines the advantages of both the flexibility of randomized
constructions and the rigidity of algebraic constructions. This method was originally proposed by Blagojevic, Karasev Bukh~\cite{Bukh2013IJM}, and then developed by Bukh~\cite{Bukh2015}. In the random algebraic method of Bukh~\cite{Bukh2015}, usually the vertex set is the vector space over finite field and the edge set is determined by a number of random polynomials. The method has been applied to complete bipartite graphs~\cite{Bukh2015}, theta graphs~\cite{Conlon2014} and several other types of hypergraphs \cite{2019He,2018Ma}. Moreover, Bukh and Conlon~\cite{Bukh2018} successfully proved the so-called rational Tur\'{a}n exponent conjecture via this method. It is worth noticing that, the random algebraic method usually requires some parameters to be very large due to Lang-Weil bound~\cite{LangWeil1954}. Bukh and Tait~\cite{2020cpcBukh} and the authors~\cite{2020XZGEUJC} studied how these Tur\'{a}n numbers depend on such large parameters of the forbidden graphs and hypergraphs. More recently, Bukh~\cite{Bukh2021} developed a novel version of random algebraic method to construct graphs that match the K\"{o}v\'{a}ri-S\'{o}s-Tur\'{a}n bound for $t$ that is only exponential in $s$.
\end{enumerate}

Our original motivation for this paper can be stated as follows.

{\textbf{Motivation}}: If we have used the random algebraic construction of Bukh~\cite{Bukh2015} to obtain tight lower bounds for Tur\'{a}n number of some given graph $H$, can we give an explicit construction of extremal graph for $H$? Furthermore, as we have mentioned above, the random algebraic method usually requires some parameters to be very large, can we obtain a better constant dependence?

As far as we know, there have been two such explicit constructions with similar motivation in recent years. Verstra\"{e}te and Williford ~\cite{VW19} gave an explicit construction of $\theta_{4,3}$-free graph with $\Omega(n^{\frac{5}{4}})$ edges, which significantly improves the constant compared with the result of Conlon~\cite{Conlon2014}. Pohoata and Zakharov~\cite{2020Norm} introduced a high uniformity generalization of the so-called (projective) norm graphs of Alon, Koll\'{a}r, R\'{o}nyai and Szab\'{o}~\cite{ARS99, KRS96}, and showed that $\text{ex}_{r}(n,K_{s_{1},s_{2},\ldots,s_{r}}^{(r)})=\Omega(n^{r-\frac{1}{s_{1}s_{2}\cdots s_{r-1}}})$ holds for all integers $s_{1},\ldots ,s_{r}\ge 2$ such that $s_{r}\ge((r-1)(s_{1}\cdots s_{r-1}-1))!+1.$ This improves upon the result of Ma, Yuan and Zhang~\cite{2018Ma}.

In this paper, we mainly focus on two forbidden structures. The first is the 1-subdivision of complete bipartite graphs. Recall that the $1$-subdivision of a graph $H$, denoted by $H'$, is the graph obtained by replacing the edges of $H$ by internally disjoint paths of length $2$. In \cite{CJL19}, Conlon, Janzer and Lee showed that $\text{ex}(n,K_{s,t}')=O(n^{\frac{3}{2}-\frac{1}{2s}})$ for $2\le s\le t$. Combining the random algebraic construction in \cite{Bukh2018}, they also showed the upper bound is tight when $t$ is sufficiently large compared to $s$. In the same paper, they asked the following problem:
\begin{problem}
For any integer $s\ge2$, estimate the smallest $t$ such that $\textup{ex}(n,K_{s,t}')=\Omega(n^{\frac{3}{2}-\frac{1}{2s}})$.
\end{problem}
 The case $s=2$ amounts to estimating the extremal number of the theta graph $\theta_{4,t}$. As we have mentioned, Verstra\"{e}te and Williford \cite{VW19} gave an algebraic construction which yields $\text{ex}(n,\theta_{4,3})=\Omega(n^{\frac{5}{4}})$. In this paper, we consider the next case $s=3$, and prove the following result.
\begin{theorem}\label{subdthm}
$\textup{ex}(n,K_{3,25}')=\Theta(n^{\frac{4}{3}})$.
\end{theorem}

The previously known smallest value of $t$ from random algebraic method in \cite{Bukh2018} is $s^{O(s^{2})}.$ In particular, when $s=3,$ their construction works for $t\approx 10^{56}.$

Secondly, we consider the linear Tur\'{a}n number of Berge theta hypergraph $\Theta_{\ell,t}^{B}.$ Define the \emph{$r$-uniform Berge theta hypergraph} $\Theta_{\ell,t}^{B}$ to be a set of distinct vertices $x,y,v_{1}^{1},\cdots,v_{\ell-1}^{1},\cdots,v_{1}^{t},$ $\cdots,v_{\ell-1}^{t}$ and a set of distinct edges $e_{1}^{1},\cdots,$ $e_{\ell}^{1},\cdots,e_{1}^{t},\cdots,e_{\ell}^{t}$ such that $\{x,v_{1}^{i}\}\subset e_{1}^{i},$ $\{v_{j-1}^{i},v_{j}^{i}\}\subset e_{j}^{i}$ and $\{v_{\ell-1}^{i},y\}\subset e_{\ell}^{i}$ for $1\le i\le t$ and $2\le j\le \ell-1.$ By contrast with the simple graph cases, only a few results are known for hypergraph Tur\'{a}n problems. A hypergraph is called \emph{linear} if every two hyperedges have at most one vertex in common. Given a family of $r$-uniform hypergraphs $\mathcal{F}$, the \emph{linear Tur\'{a}n number} of $\mathcal{F}$, denoted $\text{ex}_{r}^{lin}(n,\mathcal{F})$, is the maximum number of hyperedges in an $\mathcal{F}$-free $r$-uniform linear hypergraph on $n$ vertices. It is easy to see that $\text{ex}_{r}^{lin}(n,\mathcal{F})\le \text{ex}_{r}(n,\mathcal{F})$.

 Recently, He and Tait~\cite{2019He} gave the following upper bound \begin{align*}
\textup{ex}_{r}(n,\Theta_{\ell,t}^{B})\le c_{r,\ell,t}n^{1+\frac{1}{\ell}},
\end{align*}
where $c_{r,\ell,t}$ is a constant depending on $r,\ell,t$. They also showed that
 \begin{align*}
 \textup{ex}_{r}(n,\Theta_{\ell,t}^{B})=\Omega_{\ell,r}(n^{1+\frac{1}{\ell}}),
 \end{align*}
 where $t$ is extremely large.

 It is therefore still quite natural to ask whether for $\ell\ge 3$, there exist
improved constructions that could show that $t(\ell)$ can be chosen to be a reasonable quantity. When $\ell=3,$ we prove the following result via an explicit construction.
\begin{theorem}\label{thetathm}
$\textup{ex}_{3}^{lin}(n,\Theta_{3,217}^{B})=\Theta(n^{\frac{4}{3}})$.
\end{theorem}

Table~\ref{table1} lists some comparisons between random algebraic constructions and explicit constructions. This paper is organized as follows. In Section~\ref{sectiotech}, we will give some basics about the resultant of polynomials, which are the main techniques employed in this paper. In Section~\ref{sectionsub}, we give a proof of Theorem~\ref{subdthm}. We prove Theorem~\ref{thetathm} in Section~\ref{sectiontheta}. Section~\ref{sectionconclusion} concludes our paper. Most of the computations have been done by MAGMA \cite{BCP97}.
\begin{table}[H]
		\centering
		\begin{tabular}{|l|c|c|}\hline
			Forbidden (hyper-)graph $H$&\multicolumn{2}{c|}{Constant dependence}\\\hline
			&Random algebraic construction& Explicit construction\\
			 $\text{ex}(n,K_{s,t})=\Omega(n^{2-\frac{1}{s}})$&$t\ge s^{4s}$~\cite{Bukh2013IJM}&$t\ge (s-1)!+1$~\cite{ARS99}\\
             $\text{ex}(n,\Theta_{\ell,t})=\Omega(n^{1+\frac{1}{\ell}})$&$t=\ell^{O(\ell^{2})}$~\cite{Conlon2014}&$\ell=4,t\ge 3$~\cite{VW19}\\\
			 $\text{ex}_{r}(n,K_{s_{1},\ldots,s_{r}}^{(r)})=\Omega(n^{r-\frac{1}{s_{1}\cdots s_{r-1}}})$&$s_{r}\ge(\prod\limits_{i=1}^{r-1}s_{i})^{O(\prod\limits_{i=1}^{r-1}s_{i})}$~\cite{2018Ma}&$s_{r}\ge((r-1)(s_{1}\cdots s_{r-1}-1))!+1$~\cite{2020Norm}\\
			 $\text{ex}(n,sub(K_{s,t}))=\Omega(n^{\frac{3}{2}-\frac{1}{2s}})$&$t= s^{O(s^{2})}$~\cite{Bukh2018} &$\bm{s=3}$, $\bm{t\ge 25}$ \\
			$\text{ex}_{r}^{lin}(n,\Theta_{\ell,t}^{B})=\Omega(n^{1+\frac{1}{\ell}})$&$t=\ell^{O(\ell^{2})}$~\cite{Bukh2018}&$\bm{\ell=3,t\ge 217}$\\ \hline
			
		\end{tabular}
		\caption{Constant dependence in random algebraic constructions of Bukh~\cite{Bukh2015} and explicit constructions. }
\label{table1}
	\end{table}

\section{Main technique employed}\label{sectiotech}
Let $q$ be a prime power, and $\mathbb{F}_{q}$ be the finite field of order $q$. Let $G$ be a graph with vertex set $\mathbb{F}_{q}^{t}$ and two vertices $u,v\in\mathbb{F}_{q}^{t}$ are joined if $f_{i}(u,v)=0$ for $i=1,\dots,s$, where $f_{i}\in\mathbb{F}_{q}[x_{1},\dots,x_{2t}]$.
In order to show that $G$ does not contain certain structure, we first fix several vertices, then count the number of certain substructures containing these vertices (for example, consider theta graph $\theta_{k,t}$, we need to show that there are at most $t-1$ paths of length $k$ between any two fixed vertices, and for complete bipartite graph $K_{s,t}$, we need to show that there are at most $t-1$ common neighbours among any $s$ vertices). From these conditions, we can get a series of equations
\[F_{1}(x_{1},\dots,x_{m})=\cdots=F_{m}(x_{1},\dots,x_{m})=0,\]
(in general, there should be $m$ equations with $m$ variables), where $F_{j}$ ($j=1,\dots,m$) come from $f_{i}$ ($i=1,\dots,s$). If the construction works, then there are finitely many solutions for above equations, i.e., $\dim(W)=0$, where
\[W=\{x\in\mathbb{F}_{q}^{m}:\ F_{1}(x)=\cdots=F_{m}(x)=0\}.\]
The following lemma says that the size of $W$ is bounded by the product of degrees of $F_{i}$.
\begin{theorem}[B\'{e}zout's theorem]
Let $F_{i}\in\mathbb{F}_{q}[x_{1},\dots,x_{m}]$ for $i=1,\dots,m$. If the variety $W=\{x\in\mathbb{F}_{q}^{m}:\ F_{1}(x)=\cdots=F_{m}(x)=0\}$ has $\dim(W)=0$, then
\[|W|\le\prod_{i=1}^{m}\deg(F_{i}).\]
\end{theorem}
In general, B\'{e}zout's theorem does not give the best bound for the size of $W$. On the other hand, it is not easy to compute the dimension of $W$. But for some specific cases, we can determine the size of $W$ by eliminating variables directly.

When the forbidden graph is simple, for example, $\theta_{4,t}$ in Theorem~\ref{vw19}, we get Eqs. (\ref{vweq1})-(\ref{vweq3}), the degrees of polynomials are 2, 3 and 3. Hence this case can be done by manual calculation.
But when the forbidden graph is a bit complicated, say $\Theta_{3,t}^{B}$ in Section~\ref{sectiontheta}, we get Eqs. (\ref{thetaequation1})-(\ref{thetaequation2}), and the degrees of these polynomials are 2, 3, 2, 3, 2, 3, 2 and 3. After eliminating 7 variables, the final polynomial has degree 40 (see Eq. (\ref{thetaeq13})), this is almost impossible to get it by manual computation. Similarly for the case $K_{3,t}'$. Hence we introduce the polynomial resultant method, which can settle such cases by computer.

The polynomial resultant method has been used in \cite{GJXZ20} for Ramsey problems. For the convenience of readers, we recall some basics about the resultant of polynomials. Let $\mathbb{F}$ be a field, and $\mathbb{F}[x]$ be the polynomial ring with coefficients in $\mathbb{F}$.
\begin{definition}
Let $f(x),g(x)\in\mathbb{F}[x]$ with $f(x)=a_{m}x^{m}+\cdots+a_{1}x+a_{0}$ and $g(x)=b_{n}x^{n}+\cdots+b_{1}x+b_{0}$, then the resultant of $f$ and $g$ is defined by the determinant of the following $(m+n)\times (m+n)$ matrix,
\[
\begin{array}{c@{\hspace{-5pt}}l}
\left(
            \begin{array}{cccccccc}
              a_{0} &       &         &        & b_{0} &       &       &   \\
              a_{1} & a_{0} &         &        & b_{1} &  b_{0}&       &   \\
              a_{2} &  a_{1}& \ddots &         & b_{2} &  b_{1}& \ddots &  \\
              \vdots& a_{2} &  \ddots&   a_{0} & \vdots& b_{2} & \ddots & b_{0}\\
              a_{m} & \vdots& \ddots &a_{1}    & b_{n} & \vdots& \ddots & b_{1}  \\
                    & a_{m} & \ddots & a_{2}   &       & b_{n} &  \ddots & b_{2}  \\
                    &       &  \ddots&  \vdots &       &       & \ddots&  \vdots \\
              &       &         &  a_{m}&       &       &        & b_{n}\\
            \end{array}
\right)
          \\[-5pt]
          \begin{array}{cc}\underbrace{\rule{31mm}{0mm}}_n&
          \underbrace{\rule{31mm}{0mm}}_m\end{array}&
          \end{array},
\]
which is denoted by $R(f,g)$.
\end{definition}
The resultant of two polynomials has the following property.
\begin{lemma}\cite[Theorem 3.18]{Fuhrmann2012}\label{sublemma1}
If $\text{gcd}(f(x),g(x))=h(x)$, where $\text{deg}(h(x))\ge1$, then $R(f,g)=0$. In particular, if $f$ and $g$ have a common root in $\mathbb{F}$, then $R(f,g)=0$.
\end{lemma}

When we consider multivariable polynomials, we can define the resultant similarly, and the above lemma still holds when we fix one variable. For any $f,g\in\mathbb{F}[x_{1},\dots,x_{n}]$, let $R(f,g;x_{i})$ denote the resultant of $f$ and $g$ with respect to $x_{i}$, then we have $R(f,g;x_{i})\in\mathbb{F}[x_{1},\dots,x_{i-1},x_{i+1},x_{n}]$.
\begin{example}
We consider the following equations
\begin{align*}
&f=xy-1=0,\\
&g=x^{2}+y^{2}-4=0.
\end{align*}
Regarding $f$ and $g$ as polynomials in $x$ whose coefficients are polynomials in $y$, then we can compute to get that
\[R(f,g;x)=\text{det}\left(
                       \begin{array}{ccc}
                         y & 0 & 1 \\
                         -1 & y & 0 \\
                         0 & -1 & y^{2}-4 \\
                       \end{array}
                     \right)=y^{4}-4y^{2}+1.
\]
By Lemma~\ref{sublemma1}, it follows that $y^{4}-4y^{2}+1$ vanishes at any common solutions of $f=g=0$. Hence, by solving $y^{4}-4y^{2}+1=0$, we can find the $y$-coordinates of the solutions.
\end{example}

As an example, we give a new proof of the result stated in  the following theorem, which was originally proved by Verstra\"{e}te and Williford \cite{VW19}. Recall that a theta graph $\theta_{k,t}$ is a graph made of $t$ internally disjoint paths of length $k$ connecting two endpoints.
\begin{theorem}\cite{VW19}\label{vw19}
$\textup{ex}(n,\theta_{4,3})=\Omega(n^{\frac{5}{4}})$.
\end{theorem}
\begin{proof}
Let $q$ be an odd prime power. The graph $G_{q}$ is defined on the vertex set $V=\mathbb{F}_{q}^{4}$ such that  $u=(u_{1},u_{2},u_{3},u_{4})\in V$ is joined to $v=(v_{1},v_{2},v_{3},v_{4})\in V$ if and only if $u\ne v$ and
\begin{align*}
&u_{2}+v_{2}=u_{1}v_{1},\\
&u_{3}+v_{4}=u_{1}v_{1}^{2},\\
&u_{4}+v_{3}=u_{1}^{2}v_{1}.
\end{align*}
From above equations, it is easy to see that if $r,s\in V$ are distinct and have a common neighbour, then $r_{1}\ne s_{1}$.
It can be computed to get that $G_{q}$ has $n:=q^{4}$ vertices and $\Omega(n^{\frac{5}{4}})$ edges.

Suppose that $G_{q}$ contains a copy of $\theta_{4,3}$ with edges $\{au,uv,vw,wb,ax,xy,yz,zb,ad,de,ef,fb\}$. We first consider the octagon with edge set $\{au,uv,vw,wb,ax,xy,yz,zb\}$. By the definition of $G_{q}$, we have
\begin{align*}
&a_{2}+u_{2}=a_{1}u_{1}, && a_{3}+u_{4}=a_{1}u_{1}^{2}, && a_{4}+u_{3}=a_{1}^{2}u_{1},\\
&u_{2}+v_{2}=v_{1}u_{1}, && v_{3}+u_{4}=v_{1}u_{1}^{2}, && v_{4}+u_{3}=v_{1}^{2}u_{1},\\
&w_{2}+v_{2}=w_{1}v_{1}, && v_{3}+w_{4}=v_{1}w_{1}^{2}, && v_{4}+w_{3}=v_{1}^{2}w_{1},\\
&w_{2}+b_{2}=w_{1}b_{1}, && b_{3}+w_{4}=b_{1}w_{1}^{2}, && b_{4}+w_{3}=b_{1}^{2}w_{1},\\
&z_{2}+b_{2}=z_{1}b_{1}, && b_{3}+z_{4}=b_{1}z_{1}^{2}, && b_{4}+z_{3}=b_{1}^{2}z_{1},\\
&z_{2}+y_{2}=y_{1}z_{1}, && y_{3}+z_{4}=y_{1}z_{1}^{2}, && y_{4}+z_{3}=y_{1}^{2}z_{1},\\
&x_{2}+y_{2}=y_{1}x_{1}, && y_{3}+x_{4}=y_{1}x_{1}^{2}, && y_{4}+x_{3}=y_{1}^{2}x_{1},\\
&x_{2}+a_{2}=x_{1}a_{1}, && a_{3}+x_{4}=a_{1}x_{1}^{2}, && a_{4}+x_{3}=a_{1}^{2}x_{1}.
\end{align*}
From the left eight equations (the middle eight equations, the right eight equations, resp.), we can get the following equations.
\begin{align}
f_{1}:=a_{1}u_{1}-u_{1}v_{1}+v_{1}w_{1}-w_{1}b_{1}+b_{1}z_{1}-z_{1}y_{1}+y_{1}x_{1}-x_{1}a_{1}=0,\label{vweq1}\\
f_{2}:=a_{1}u_{1}^{2}-u_{1}^{2}v_{1}+v_{1}w_{1}^{2}-w_{1}^{2}b_{1}+b_{1}z_{1}^{2}-z_{1}^{2}y_{1}+y_{1}x_{1}^{2}-x_{1}^{2}a_{1}=0,\label{vweq2}\\
f_{3}:=a_{1}^{2}u_{1}-u_{1}v_{1}^{2}+v_{1}^{2}w_{1}-w_{1}b_{1}^{2}+b_{1}^{2}z_{1}-z_{1}y_{1}^{2}+y_{1}^{2}x_{1}-x_{1}a_{1}^{2}=0.\label{vweq3}
\end{align}

Regarding $f_{i}$ as polynomials with variable $u_{1}$, we can compute to get that
\begin{align*}
&R(f_{1},f_{2};u_{1})=g_{1}(a_{1}-v_{1}),\\
&R(f_{1},f_{3};u_{1})=g_{2}(a_{1}-v_{1}),
\end{align*}
where
\begin{align*}
g_{1}=&-a_{1}b_{1}w_{1}^2 + 2a_{1}b_{1}w_{1}x_{1} - 2a_{1}b_{1}x_{1}z_{1} + a_{1}b_{1}z_{1}^2 + a_{1}v_{1}w_{1}^2 - 2a_{1}v_{1}w_{1}x_{1} + a_{1}v_{1}x_{1}^2 - a_{1}x_{1}^2y_{1} + \\
&2a_{1}x_{1}y_{1}z_{1} - a_{1}y_{1}z_{1}^2 + b_{1}^2w_{1}^2 -
    2b_{1}^2w_{1}z_{1} + b_{1}^2z_{1}^2 - b_{1}v_{1}w_{1}^2 + 2b_{1}v_{1}w_{1}z_{1} - b_{1}v_{1}z_{1}^2 - 2b_{1}w_{1}x_{1}y_{1} + \\&2b_{1}w_{1}y_{1}z_{1} + 2b_{1}x_{1}y_{1}z_{1} - 2b_{1}y_{1}z_{1}^2 + 2v_{1}w_{1}x_{1}y_{1} -
    2v_{1}w_{1}y_{1}z_{1} - v_{1}x_{1}^2y_{1} + v_{1}y_{1}z_{1}^2 + x_{1}^2y_{1}^2 -\\& 2x_{1}y_{1}^2z_{1} + y_{1}^2z_{1}^2,\\
g_{2}=&a_{1}b_{1}w_{1} - a_{1}b_{1}z_{1} - a_{1}v_{1}w_{1} + a_{1}v_{1}x_{1} - a_{1}x_{1}y_{1} + a_{1}y_{1}z_{1} - b_{1}^2w_{1} + b_{1}^2z_{1} + b_{1}v_{1}w_{1} - b_{1}v_{1}z_{1} -\\& v_{1}x_{1}y_{1} + v_{1}y_{1}z_{1} + x_{1}y_{1}^2 - y_{1}^2z_{1}.
\end{align*}
Since $a$ and $v$ have a common neighbour, then $a_{1}\ne v_{1}$. Hence $g_{1}=g_{2}=0$.
We can regard $g_{1}$ and $g_{2}$ as polynomials with variable $x_{1}$, then we have
\begin{align*}
R(g_{1},g_{2};x_{1})=(b_{1}-y_{1})(w_{1}-z_{1})^{2}(v_{1}-y_{1})(v_{1}-b_{1})(a_{1}-y_{1})(a_{1}-b_{1})(a_{1}-v_{1}+b_{1}-y_{1}).
\end{align*}
Note that $b$ and $y$ ($w$ and $z$, $v$ and $b$, $a$ and $y$, resp.) have a common neighbour, thus we have either $a_{1}=b_{1}$, or $v_{1}=y_{1}$, or $a_{1}+b_{1}=v_{1}+y_{1}$.

If $a_{1}=b_{1}$, substituting $b_{1}$ by $a_{1}$ into $f_{1}$ and $f_{2}$, we obtain that $R(f_{1},f_{2};u_{1})=(z_{1}-x_{1})(v_{1}-y_{1})(a_{1}-y_{1})(a_{1}-v_{1})$. Since $x$ and $z$ ($a$ and $y$, $a$ and $v$, resp.) have a common neighbour, then we must have $v_{1}=y_{1}$. Substituting $a_{1}=b_{1}$ and $v_{1}=y_{1}$ into $f_{1}$ and $f_{3}$, then we get $R(f_{1},f_{3};u_{1})=(z_{1}-x_{1})(w_{1}-z_{1})(a_{1}-v_{1})^{3}=0$, which contradicts the fact that $z$ and $x$ ($w$ and $z$, $a$ and $v$, resp.) have a common neighbour. Hence $a_{1}\ne b_{1}$. Similarly, $v_{1}\ne y_{1}$.

Therefore, we have
\begin{align*}a_{1}+b_{1}=v_{1}+y_{1}.\end{align*}
 By the symmetry of octagon, we also have
 \begin{align}
 u_{1}+z_{1}=x_{1}+w_{1}.\label{eq100}
  \end{align}
  Since there are three octagons $\{au,uv,vw,wb,ax,xy,yz,zb\}$, $\{au,uv,vw,wb,ad,de,ef,fb\}$, $\{ax,xy,yz,\\zb,ad,de,ef,fb\}$ in $\theta_{4,3}$, we have $a_{1}+b_{1}=v_{1}+y_{1}=v_{1}+e_{1}=y_{1}+e_{1}$. Hence $v_{1}=y_{1}=e_{1}$ and $a_{1}+b_{1}=2v_{1}$. Substituting these equations into $f_{1}$,
we have $(a_{1}-v_{1})(u_{1}+w_{1}-x_{1}-z_{1})=0.$
Since $a,v$ have a common neighbour, then $a_{1}\ne v_{1}$. Hence we have $u_{1}+w_{1}=x_{1}+z_{1}$. Combining with Eq.~(\ref{eq100}), we get that $u_{1}=x_{1}$, contradicts the fact that $u,x$ have a common neighbour $a$.
\end{proof}

\section{Constructions of $K_{3,25}'$-free graphs}\label{sectionsub}
In this section, we construct a $K_{3,25}'$-free graph with $n$ vertices and $\Omega(n^{\frac{4}{3}})$ edges.
Let $\mathbb{F}_{p}$ be a finite field, where $p$ is an odd prime with $p\equiv5\pmod{6}$ and $p>23$. Let
\[S=\bigg\{x: x\in \bigg\{1,2,\dots,\frac{p-5}{6}\bigg\}\text{ and }x\equiv1\pmod{3}\bigg\}\subseteq\mathbb{F}_{p}.\]
Then we have the following lemma.
\begin{lemma}\label{subdlemma3}
For any $x,y,z,t\in S$, we have $x+y\ne0$, $x+y+z+t\ne0$, $x+5y\ne0$, and $x^{2}+xy+y^{2}\ne0$.
\end{lemma}
\begin{proof}
Since $x,y,z,t\in \mathbb{F}_{p}$ and $x,y,z,t\in\{1,2,\dots,\frac{p-5}{6}\}$, we have $x+y\ne0$, $x+y+z+t\ne0$, $x+5y\ne0$.
Note that $p\equiv5\pmod6$, then $-3$ is a non-quadratic residue module $p$. Hence, $x^{2}+xy+y^{2}=(x+\frac{y}{2})^{2}+3(\frac{y}{2})^{2}\ne0$.
\end{proof}

Now we provide our construction as follows.
\begin{construction}
 The graph $G_{p}$ is defined with vertex set $V:=S\times\mathbb{F}_{p}\times\mathbb{F}_{p}$, where $x=(x_{1},x_{2},x_{3})\in V$ is joined to $y=(y_{1},y_{2},y_{3})\in V$ if $x\ne y$ and
\begin{align*}
&x_{2}+y_{3}=x_{1}y_{1}^{2},\\
&x_{3}+y_{2}=x_{1}^{2}y_{1}.
\end{align*}
\end{construction}
For any vertex $x=(x_{1},x_{2},x_{3})\in V$, from above two equations, any choice of $y_{1}\in S$ can uniquely determine $y_{2}$ and $y_{3}$. Note that $x$ should be different from $y$, then each vertex has degree at least $\frac{p-5}{18}-1=\frac{p-23}{18}$. Hence $G_{p}$ has $n:=\lceil\frac{p-5}{18}\rceil p^{2}$ vertices and at least $\frac{1}{18\times18\times2}(p-23)(p-5)p^{2}=\Omega(n^{\frac{4}{3}})$ edges. In the following of this section, we will prove that $G_{p}$ is $K_{3,25}'$-free.
We begin with the following simple lemma.
\begin{lemma}\label{subdlemma1}
If $x,y\in V$ are distinct and have a common neighbour, then $x_{1}\ne y_{1}$, $x_{2}\ne y_{2}$ and $x_{3}\ne y_{3}$.
\end{lemma}
\begin{proof}
Suppose $x,y$ have a common neighbour $u$, then
\begin{align}
&x_{2}+u_{3}=x_{1}u_{1}^{2},\label{subeq1}\\
&x_{3}+u_{2}=x_{1}^{2}u_{1},\label{subeq2}\\
&y_{2}+u_{3}=y_{1}u_{1}^{2},\label{subeq3}\\
&y_{3}+u_{2}=y_{1}^{2}u_{1}.\label{subeq4}
\end{align}
Now we divide our discussion into three cases.

{\bf{Case 1: $x_{1}=y_{1}.$}}

For this case, from Eqs. (\ref{subeq1}) and (\ref{subeq3}), we have $x_{2}=y_{2}$. Similarly, from Eqs. (\ref{subeq2}) and (\ref{subeq4}), we have $x_{3}=y_{3}$. Hence $x=y$, which is a contradiction.

{\bf{Case 2: $x_{2}=y_{2}.$}}

For this case, from Eqs. (\ref{subeq1}) and (\ref{subeq3}), we have $(x_{1}-y_{1})u_{1}^{2}=0$. Then $x_{1}=y_{1}$, this has been done in Case 1.

{\bf{Case 3: $x_{3}=y_{3}.$}}

For this case, from Eqs. (\ref{subeq2}) and (\ref{subeq4}), we have $u_{1}(x_{1}^{2}-y_{1}^{2})=0$. Hence either $x_{1}=y_{1}$ or $x_{1}+y_{1}=0$ Hence at least one of $x_{1}=y_{1}$ and $x_{1}+y_{1}=0$ occurs. If $x_{1}=y_{1}$, this has been done in Case 1. If $x_{1}+y_{1}=0$, this contradicts Lemma~\ref{subdlemma3}.
\end{proof}
For any given $a,b,c\in V$ with $a,b,c$ being pairwise distinct, we estimate the number of sequences $(x,y,z, w)\in V^{4}$ with $x,y,z,w$ being pairwise distinct, such that $ax,xw,by,yw,cz,zw$ are edges in $G_{p}$. We will prove that there are at most $24$ different such sequences. By the definition of graph $G_{p}$, we have
\begin{align*}
&a_{2}+x_{3}=a_{1}x_{1}^{2}, && a_{3}+x_{2}=a_{1}^{2}x_{1},\\
&w_{2}+x_{3}=w_{1}x_{1}^{2}, && w_{3}+x_{2}=w_{1}^{2}x_{1},\\
&b_{2}+y_{3}=b_{1}y_{1}^{2}, && b_{3}+y_{2}=b_{1}^{2}y_{1},\\
&w_{2}+y_{3}=w_{1}y_{1}^{2}, && w_{3}+y_{2}=w_{1}^{2}y_{1},\\
&c_{2}+z_{3}=c_{1}z_{1}^{2}, && c_{3}+z_{2}=c_{1}^{2}z_{1},\\
&w_{2}+z_{3}=w_{1}z_{1}^{2}, && w_{3}+z_{2}=w_{1}^{2}z_{1}.
\end{align*}
Cancelling $x_{2},x_{3},y_{2},y_{3},z_{2},z_{3}$ from the above equations, we can get the following equations
\begin{align}
&f_{1}:=a_{2}-w_{2}-x_{1}^{2}(a_{1}-w_{1})=0,\label{subdeq1}\\
&f_{2}:=a_{3}-w_{3}-x_{1}(a_{1}^{2}-w_{1}^{2})=0,\label{subdeq2}\\
&f_{3}:=b_{2}-w_{2}-y_{1}^{2}(b_{1}-w_{1})=0,\label{subdeq3}\\
&f_{4}:=b_{3}-w_{3}-y_{1}(b_{1}^{2}-w_{1}^{2})=0,\label{subdeq4}\\
&f_{5}:=c_{2}-w_{2}-z_{1}^{2}(c_{1}-w_{1})=0,\label{subdeq5}\\
&f_{6}:=c_{3}-w_{3}-z_{1}(c_{1}^{2}-w_{1}^{2})=0.\label{subdeq6}
\end{align}
In the following of this section, we divide our discussions into three parts addressed in the following subsections in sequence.
\subsection{$a_{1}=b_{1}$ or $a_{1}=c_{1}$ or $b_{1}=c_{1}$}\label{subsec1}
We begin with the following lemma.
\begin{lemma}\label{subdlemma4}\
\begin{enumerate}
  \item[(1)] If $a_{1}=b_{1}$, then $a_{2}\ne b_{2}$ and $a_{3}\ne b_{3}$.
  \item[(2)] If $a_{1}=c_{1}$, then $a_{2}\ne c_{2}$ and $a_{3}\ne c_{3}$.
  \item[(3)] If $b_{1}=c_{1}$, then $b_{2}\ne c_{2}$ and $b_{3}\ne c_{3}$.
\end{enumerate}
\end{lemma}
\begin{proof}
We will only prove the first statement, the proofs of the other two are similar.

If $a_{2}=b_{2}$, then by Eqs. (\ref{subdeq1}) and (\ref{subdeq3}), we have $x_{1}^{2}=y_{1}^{2}$, hence either $x_{1}=y_{1}$ or $x_{1}=-y_{1}$, which contradicts Lemmas~\ref{subdlemma3} and \ref{subdlemma1}.

If $a_{3}=b_{3}$, then by Eqs. (\ref{subdeq2}) and (\ref{subdeq4}), we have $x_{1}=y_{1}$, which contradicts Lemma~\ref{subdlemma3}.
\end{proof}

Without loss of generality, we assume that $a_{1}=b_{1}$. Then Eqs. (\ref{subdeq1})-(\ref{subdeq6}) become
\begin{align*}
&f_{1}:=a_{2}-w_{2}-x_{1}^{2}(a_{1}-w_{1})=0,\\
&f_{2}:=a_{3}-w_{3}-x_{1}(a_{1}^{2}-w_{1}^{2})=0,\\
&f_{3}:=b_{2}-w_{2}-y_{1}^{2}(a_{1}-w_{1})=0,\\
&f_{4}:=b_{3}-w_{3}-y_{1}(a_{1}^{2}-w_{1}^{2})=0,\\
&f_{5}:=c_{2}-w_{2}-z_{1}^{2}(c_{1}-w_{1})=0,\\
&f_{6}:=c_{3}-w_{3}-z_{1}(c_{1}^{2}-w_{1}^{2})=0.
\end{align*}

Now we regard $f_{i}$ ($i=1,2,\dots,6$) as polynomials with variables $x_{1},y_{1},z_{1},w_{1},w_{2},w_{3}$. By the MAGMA Program 1 in the Appendix A, we can get that
\begin{align*}
&R(f_{1},f_{2};x_{1})=g_{1}\cdot(a_{1}-w_{1}),\\
&R(f_{3},f_{4};y_{1})=g_{2}\cdot(a_{1}-w_{1}),\\
&R(f_{5},f_{6};z_{1})=g_{3}\cdot(c_{1}-w_{1}).
\end{align*}
By Lemma~\ref{subdlemma1}, we have $a_{1}\ne w_{1}$ and $c_{1}\ne w_{1}$. Then $g_{1}=g_{2}=g_{3}=0$ and
\begin{align*}
&R(g_{1},g_{2};w_{2})=g_{4}\cdot(a_{1}-w_{1})(a_{1}+w_{1})^{2},\\
&R(g_{1},g_{3};w_{2})=g_{5}.
\end{align*}
By Lemmas~\ref{subdlemma3} and \ref{subdlemma1}, $a_{1}+w_{1}\ne0$ and $a_{1}-w_{1}\ne0$, then $g_{4}=g_{5}=0$. Let $h=R(g_{4},g_{5};w_{3})$, then $h$ is a polynomial of $w_{1}$ with degree $8$. We can write $h$ as $h=\sum_{i=0}^{8}h_{i}w_{1}^{i}$. Then we can compute to get that
\begin{align*}
h_{8}=-(a_{2}-b_{2})^{2}(a_{1}-c_{1}).
\end{align*}
By Lemma~\ref{subdlemma4}, $a_{2}\ne b_{2}$.
If $a_{1}\ne c_{1}$, then there are at most 8 solutions for $w_{1}$. For any fixed $w_{1}$, $g_{4}$ is a polynomial of $w_{3}$ with degree 1. We write $g_{4}$ as $g_{4}=s_{1}w_{3}+s_{0}$, where $s_{1}=a_{3}-b_{3}\ne0$. Hence there is at most $1$ solution for $w_{3}$. If $w_{1}$ and $w_{3}$ are given, then all the remaining variables are uniquely determined. Hence there are at most $8\ (=8\times1)$ different sequences of $(x,y,z,w)$ for this case.

If $a_{1}=c_{1}$, then we can compute to get that $g_{5}=g_{6}\cdot(a_{1}-w_{1})(a_{1}+w_{1})^{2}$. Let $h'=R(g_{4},g_{6};w_{3})$, then $h'$ is a polynomial of $w_{1}$ with degree $3$.  We can write $h'$ as $h'=\sum_{i=0}^{3}k_{i}w_{1}^{i}$. Regarding $k_{i}$ ($i=0,1,2,3$) as polynomials with variable $a_{2}$, then by the MAGMA Program 2 in the Appendix A, we have
\begin{align*}
R(k_{0},k_{3};a_{2})=4(b_{3}-c_{3})^{2}(a_{3}-c_{3})(a_{3}-b_{3}).
\end{align*}
By Lemma~\ref{subdlemma4}, $R(k_{0},k_{3};a_{2})\ne0$, then $k_{0}$ and $k_{3}$ cannot be both zero. Hence there are at most 3 solutions for $w_{1}$. Similarly as above, for any fixed $w_{1}$, there is at most $1$ solution for $w_{3}$. If $w_{1}$ and $w_{3}$ are given, then all the remaining variables are uniquely determined. Hence there are at most $3\ (=3\times1)$ different sequences of $(x,y,z,w)$ for this case.

Therefore, if $a_{1}=b_{1}$ or $a_{1}=c_{1}$ or $b_{1}=c_{1}$, then there are at most $8$ different sequences of $(x,y,z,w)$.

\subsection{$a_{1}\ne b_{1}$, $a_{1}\ne c_{1}$, $b_{1}\ne c_{1}$ and $a_{2}=b_{2}$ or $a_{2}=c_{2}$ or $b_{2}=c_{2}$}\label{subsec2}
Without loss of generality, we assume that $a_{2}=b_{2}$. Then Eqs. (\ref{subdeq1})-(\ref{subdeq6}) become
\begin{align*}
&f_{1}:=a_{2}-w_{2}-x_{1}^{2}(a_{1}-w_{1})=0,\\
&f_{2}:=a_{3}-w_{3}-x_{1}(a_{1}^{2}-w_{1}^{2})=0,\\
&f_{3}:=a_{2}-w_{2}-y_{1}^{2}(b_{1}-w_{1})=0,\\
&f_{4}:=b_{3}-w_{3}-y_{1}(b_{1}^{2}-w_{1}^{2})=0,\\
&f_{5}:=c_{2}-w_{2}-z_{1}^{2}(c_{1}-w_{1})=0,\\
&f_{6}:=c_{3}-w_{3}-z_{1}(c_{1}^{2}-w_{1}^{2})=0.
\end{align*}
Now we regard $f_{i}$ ($i=1,2,\dots,6$) as polynomials with variables $x_{1},y_{1},z_{1},w_{1},w_{2},w_{3}$, respectively. By the MAGMA Program 3 in the Appendix A, we can get that
\begin{align*}
&R(f_{1},f_{2};x_{1})=g_{1}\cdot(a_{1}-w_{1}),\\
&R(f_{3},f_{4};y_{1})=g_{2}\cdot(b_{1}-w_{1}),\\
&R(f_{5},f_{6};z_{1})=g_{3}\cdot(c_{1}-w_{1}).
\end{align*}
Note that $a$ and $w$ ($b$ and $w$, $c$ and $w$, resp.) have a common neighbour, by Lemma~\ref{subdlemma1}, we have $a_{1}\ne w_{1}$ ($b_{1}\ne w_{1}$, $c_{1}\ne w_{1}$, resp.). Then $g_{1}=g_{2}=g_{3}=0$ and
\begin{align*}
&R(g_{1},g_{2};w_{2})=g_{4},\\
&R(g_{1},g_{3};w_{2})=g_{5},\\
&R(g_{4},g_{5};w_{3})=h\cdot (a_{1}-w_{1})^{2}(a_{1}+w_{1})^{4},
\end{align*}
where $h$ is a polynomial of $w_{1}$ with degree $10$. By Lemmas~\ref{subdlemma3} and \ref{subdlemma1}, $a_{1}\ne w_{1}$ and $a_{1}+w_{1}\ne0$, then $h=0$. We can write $h$ as $h=\sum_{i=0}^{10}h_{i}w_{1}^{i}$, where
\begin{align*}
h_{10}=(a_{2}-c_{2})^{2}(a_{1}-b_{1})^{2}.
\end{align*}

If $a_{2}\ne c_{2}$, then $h_{10}\ne0$. Hence there are at most 10 solutions for $w_{1}$. For any fixed $w_{1}$, $g_{4}$ and $g_{5}$ are polynomials of $w_{3}$ with degree 2. We write $g_{4}$ and $g_{5}$ as $g_{4}=\sum_{i=0}^{2}s_{i}w_{3}^{i}$ and $g_{5}=\sum_{i=0}^{2}t_{i}w_{3}^{i}$, then $s_{2}=s_{2}'\cdot(a_{1}-b_{1})$ and $t_{2}=t_{2}'\cdot(a_{1}-c_{1})$. We can compute to get that $s_{2}'-t_{2}'=(b_{1}-c_{1})(a_{1}+b_{1}+c_{1}+w_{1})$. By Lemmas~\ref{subdlemma3} and \ref{subdlemma1}, we have $s_{2}'-t_{2}'\ne0$. Hence there is at least one of $s_{2},t_{2}$ not being 0, then there are at most $2$ solutions for $w_{3}$. If $w_{1}$ and $w_{3}$ are given, then all the remaining variables are uniquely determined. Hence there are at most $20\ (=10\times2)$ different sequences of $(x,y,z,w)$ for this case.

If $a_{2}=c_{2}$, then $h_{i}=0$ for $i=6,7,8,9,10$. We can regard $h_{4}$ and $h_{5}$ as polynomials with variable $a_{1}$. Then by the MAGMA Program 4 in the Appendix A, we have
\begin{align*}
R(h_{4},h_{5};a_{1})=80(b_{1}-c_{1})^{2}(a_{3}-b_{3})^{4}(a_{3}-c_{3})^{4}(b_{3}-c_{3})^{4}.
\end{align*}

If $R(h_{4},h_{5};a_{1})\ne0$, then at least one of $h_{4},h_{5}$ is not 0. Hence there are at most 5 solutions for $w_{1}$. For any fixed $w_{1}$, through a similar discussion as above, there are at most $2$ solutions for $w_{3}$. If $w_{1}$ and $w_{3}$ are given, then all the remaining variables are uniquely determined. Hence there are at most $10\ (=5\times2)$ different sequences of $(x,y,z,w)$ for this case.

If $R(h_{4},h_{5};a_{1})=0$, without loss of generality, we assume that $a_{3}=b_{3}$.  Then by the MAGMA Program 5 in the Appendix A, we have
\begin{align*}
g_{4}=(a_{3}-w_{3})^{2}(a_{1}-b_{1})\cdot g_{4}',
\end{align*}
where $g_{4}'=-w_{1}^{2}+(a_{1}+b_{1})w_{1}+a_{1}^{2}+a_{1}b_{1}+b_{1}^{2}$. By Eq.~(\ref{subdeq2}), if $a_{3}=w_{3}$, then $a_{1}^{2}=w_{1}^{2}$, which is a contradiction. Hence $a_{3}\ne w_{3}$ and $a_{1}\ne b_{1}$, then there are at most 2 solutions for $w_{1}$. For any fixed $w_{1}$, $g_{5}$ is a polynomial of $w_{3}$ with degree 2. We write $g_{5}$ as $g_{5}=\sum_{i=0}^{2}t_{i}w_{3}^{i}$, then $t_{2}=t_{2}'\cdot(a_{1}-c_{1})$. We can compute to get that $g_{4}'-t_{2}'=(b_{1}-c_{1})(a_{1}+b_{1}+c_{1}+w_{1})$. By Lemmas~\ref{subdlemma3} and \ref{subdlemma1}, we have $g_{4}'-t_{2}'\ne0$. Hence there are at most $2$ solutions for $w_{3}$. If $w_{1}$ and $w_{3}$ are given, then all the remaining variables are uniquely determined. Hence there are at most $4\ (=2\times2)$ different sequences of $(x,y,z,w)$ for this case.

Therefore, if $a_{1}\ne b_{1}$, $a_{1}\ne c_{1}$, $b_{1}\ne c_{1}$ and $a_{2}=b_{2}$ or $a_{2}=c_{2}$ or $b_{2}=c_{2}$, then there are at most $20$ different sequences of $(x,y,z,w)$.

\subsection{$a_{1}\ne b_{1}$, $a_{1}\ne c_{1}$, $b_{1}\ne c_{1}$, $a_{2}\ne b_{2}$, $a_{2}\ne c_{2}$ and $b_{2}\ne c_{2}$}\label{subsec3}
For this case, we begin with the following lemma.
\begin{lemma}\label{subdlemma2}
If $G_{p}$ contains a copy of $\theta_{3,3}$, which consists of the set of edges $\{da,db,dc,ax,by,cz,\\wx,wy,wz\}$, then $a_{1}y_{1}-a_{1}z_{1}-b_{1}x_{1}+b_{1}z_{1}+c_{1}x_{1}-c_{1}y_{1}=0$.
\end{lemma}
\begin{proof}
We first consider the edges $da,ax,xw,wy,yb,bd$, which form a hexagon. By the definition of $G_{p}$, we have
\begin{align}
&d_{3}+a_{2}=d_{1}^{2}a_{1},&&a_{3}+d_{2}=a_{1}^{2}d_{1},\label{subdeq7}\\
&a_{2}+x_{3}=a_{1}x_{1}^{2},&&x_{2}+a_{3}=x_{1}a_{1}^{2},\label{subdeq8}\\
&x_{3}+w_{2}=x_{1}^{2}w_{1},&&w_{3}+x_{2}=w_{1}^{2}x_{1},\label{subdeq9}\\
&w_{2}+y_{3}=w_{1}y_{1}^{2},&&y_{2}+w_{3}=y_{1}w_{1}^{2},\label{subdeq10}\\
&y_{3}+b_{2}=y_{1}^{2}b_{1},&&b_{3}+y_{2}=b_{1}^{2}y_{1},\label{subdeq11}\\
&b_{2}+d_{3}=b_{1}d_{1}^{2},&&d_{2}+b_{3}=d_{1}b_{1}^{2}.\label{subdeq12}
\end{align}
Then we can compute to get that
\begin{align*}
&f_{1}:=d_{1}^{2}a_{1}-a_{1}x_{1}^{2}+x_{1}^{2}w_{1}-w_{1}y_{1}^{2}+y_{1}^{2}b_{1}-b_{1}d_{1}^{2}=0,\\
&f_{2}:=d_{1}a_{1}^{2}-a_{1}^{2}x_{1}+x_{1}w_{1}^{2}-w_{1}^{2}y_{1}+y_{1}b_{1}^{2}-b_{1}^{2}d_{1}=0,
\end{align*}
where $f_{1}$ is from the left six equations of (\ref{subdeq7})-(\ref{subdeq12}) and $f_{2}$ is from the right six equations of (\ref{subdeq7})-(\ref{subdeq12}). Regarding $f_{1},f_{2}$ as polynomials with variables $a_{1},u_{1},v_{1},b_{1},x_{1},w_{1}$, we can compute to get that
\begin{align*}
R(f_{1},f_{2};b_{1})=&(a_{1}-w_{1})(x_{1}-y_{1})(d_{1}-y_{1})(d_{1}-x_{1})(d_{1}^{2}a_{1}+d_{1}^{2}w_{1}-d_{1}a_{1}x_{1}-d_{1}a_{1}y_{1}+\\
&d_{1}x_{1}w_{1}+d_{1}w_{1}y_{1}-a_{1}x_{1}^{2}-a_{1}x_{1}y_{1}-a_{1}y_{1}^{2}+x_{1}^{2}w_{1}+x_{1}w_{1}y_{1}-w_{1}y_{1}^{2}).
\end{align*}
By Lemma~\ref{subdlemma1} and $f_{1}=f_{2}=0$, we have
 \begin{align*}
 d_{1}^{2}a_{1}+d_{1}^{2}w_{1}-d_{1}a_{1}x_{1}-d_{1}a_{1}y_{1}+d_{1}x_{1}w_{1}+d_{1}w_{1}y_{1}-a_{1}x_{1}^{2}-a_{1}x_{1}y_{1}-
 a_{1}y_{1}^{2}+x_{1}^{2}w_{1}+x_{1}w_{1}y_{1}-w_{1}y_{1}^{2}=0.
 \end{align*}
 Similarly, the edges $da,ax,xw,wz,zc,cd$ also form a hexagon, we have
 \begin{align*}
 d_{1}^{2}a_{1}+d_{1}^{2}w_{1}-d_{1}a_{1}x_{1}-d_{1}a_{1}z_{1}+d_{1}x_{1}w_{1}+d_{1}w_{1}z_{1}-a_{1}x_{1}^{2}-a_{1}x_{1}z_{1}-
 a_{1}z_{1}^{2}+x_{1}^{2}w_{1}+x_{1}w_{1}z_{1}-w_{1}z_{1}^{2}=0.
 \end{align*}
From the above two equations, we get
\begin{align*}
(d_{1}+x_{1})(w_{1}-a_{1})y_{1}-(w_{1}+a_{1})y_{1}^{2}=(d_{1}+x_{1})(w_{1}-a_{1})z_{1}-(w_{1}+a_{1})z_{1}^{2}.
\end{align*}
Then we can compute to get that
\begin{align*}
\frac{d_{1}+x_{1}}{w_{1}+a_{1}}=\frac{y_{1}+z_{1}}{w_{1}-a_{1}}.
\end{align*}
By the symmetry of $\theta_{3,3}$, we have
\begin{align*}
\frac{w_{1}+a_{1}}{d_{1}+x_{1}}=\frac{b_{1}+c_{1}}{d_{1}-x_{1}}.
\end{align*}
From the above two equations, we can get
\begin{align*}
g_{1}=(y_{1}+z_{1})(b_{1}+c_{1})-(w_{1}-a_{1})(d_{1}-x_{1})=0.
\end{align*}
By the symmetry of $\theta_{3,3}$ again, we also have
\begin{align*}
&g_{2}=(x_{1}+z_{1})(a_{1}+c_{1})-(w_{1}-b_{1})(d_{1}-y_{1})=0,\\
&g_{3}=(x_{1}+y_{1})(a_{1}+b_{1})-(w_{1}-c_{1})(d_{1}-z_{1})=0.
\end{align*}
Now we regard $g_{i}$ ($i=1,2,3$) as polynomials with variables $d_{1},w_{1}$. By the MAGMA Program 6 in the Appendix A, we can get that
\begin{align*}
&R(g_{1},g_{2};d_{1})=h_{1}\cdot(a_{1}+b_{1}+c_{1}-w_{1}),\\
&R(g_{1},g_{3};d_{1})=h_{2}\cdot(a_{1}+b_{1}+c_{1}-w_{1}).
\end{align*}
Note that $a_{1},b_{1},c_{1},w_{1}\in S$, then $w_{1}\in\{1,2,\dots,\frac{p-5}{6}\}$, $w_{1}\equiv1\pmod3$, $a_{1}+b_{1}+c_{1}\in\{3,\dots,\frac{p-5}{2}\}$ and $a_{1}+b_{1}+c_{1}\equiv0\pmod{3}$. Hence $a_{1}+b_{1}+c_{1}-w_{1}\ne0$. Then $h_{1}=h_{2}=0$, we can compute to get that $R(h_{1},h_{2};w_{1})=(y_{1}+z_{1})(a_{1}y_{1}-a_{1}z_{1}-b_{1}x_{1}+b_{1}z_{1}+c_{1}x_{1}-c_{1}y_{1})=0$. Hence $a_{1}y_{1}-a_{1}z_{1}-b_{1}x_{1}+b_{1}z_{1}+c_{1}x_{1}-c_{1}y_{1}=0$.
\end{proof}

\begin{remark}\label{subdrmk1}
In the above proof, we do not need the condition $d_{1}\in S$ for $d\in V$. Hence, if $d_{1}\in\mathbb{F}_{p}^{*}\backslash S$, then Lemma~\ref{subdlemma2} still holds.
\end{remark}

By Lemma~\ref{subdlemma1}, if $d,a,b,c,x,y,z,w$ form a copy of $\theta_{3,3}$ with edge set $\{da,db,dc,ax,by,cz,wx,\\wy,wz\}$, then $a_{3}\ne b_{3}$.
\begin{claim}\label{subdclaim6}
If $a_{1}y_{1}-a_{1}z_{1}-b_{1}x_{1}+b_{1}z_{1}+c_{1}x_{1}-c_{1}y_{1}=0$ and $a_{3}\ne b_{3}$, then there are at most $24$ different sequences of $(x,y,z,w)$.
\end{claim}
\begin{proof}
Let $f_{7}=a_{1}y_{1}-a_{1}z_{1}-b_{1}x_{1}+b_{1}z_{1}+c_{1}x_{1}-c_{1}y_{1}=0$, then we regard $f_{7}$ and $f_{i}$ ($i=1,2,\dots,6$) in Eqs. (\ref{subdeq1})-(\ref{subdeq6}) as polynomials with variables $x_{1},y_{1},z_{1},w_{1},w_{2},w_{3}$. Let $g_{1}=f_{1}-f_{3}$, $g_{2}=f_{1}-f_{5}$ and $g_{3}=f_{2}-f_{4}$. By the MAGMA Program 7 in the Appendix A, we can get that
\begin{align*}
&R(g_{1},g_{2};w_{1})=g_{5},\\
&R(g_{1},g_{3};w_{1})=g_{6}\cdot(x_{1}-y_{1}),\\
&R(f_{7},g_{5};x_{1})=h_{1}\cdot(y_{1}-z_{1}),\\
&R(f_{7},g_{6};x_{1})=h_{2},\\
&R(h_{1},h_{2};z_{1})=(b_{1}-c_{1})^{3}(a_{1}-b_{1})^{6}(a_{1}y_{1}^{2}-b_{1}y_{1}^{2}-a_{2}+b_{2})^{2}\cdot s,
\end{align*}
where $s$ is a polynomial of $y_{1}$ with degree $8$.

If $a_{1}y_{1}^{2}-b_{1}y_{1}^{2}-a_{2}+b_{2}=0$, by Eqs. (\ref{subdeq1}) and (\ref{subdeq3}), we have
 \begin{align*}
 0&=f_{1}-f_{3}\\
 &=a_{2}-b_{2}-x_{1}^{2}(a_{1}-w_{1})+y_{1}^{2}(b_{1}-w_{1})\\
 &=a_{2}-b_{2}-(x_{1}^{2}-y_{1}^{2})(a_{1}-w_{1})+y_{1}^{2}(b_{1}-a_{1}).
 \end{align*}
Then we have $(x_{1}^{2}-y_{1}^{2})(a_{1}-w_{1})=0$, which is a contradiction. Hence $a_{1}y_{1}^{2}-b_{1}y_{1}^{2}-a_{2}+b_{2}\ne0$.

 We write $s=\sum_{i=0}^{8}s_{i}y_{1}^{i}$, then we can compute to get that $s_{8}=s_{8}'\cdot(b_{1}-c_{1})^{3}(a_{1}-c_{1})^{4}(a_{1}-b_{1})^{4}$ and $s_{7}=s_{7}'\cdot(b_{1}-c_{1})^{3}(a_{3}-b_{3})(a_{1}-c_{1})^{3}(a_{1}-b_{1})^{3}$. We regard $s_{8}'$ and $s_{7}'$ as polynomials of $a_{1}$, then by Lemma~\ref{subdlemma3}, we have $R(s_{7}',s_{8}';a_{1})=(b_{1}+c_{1})(b_{1}+5c_{1})(b_{1}^{2}+b_{1}c_{1}+c_{1}^{2})\ne0$. Hence there are at most $8$ solutions for $y_{1}$. For any given $y_{1}$, $h_{1}$ is a polynomial of $z_{1}$ with degree 3. We write $h_{1}$ as $h_{1}=\sum_{i=0}^{3}s_{i}z_{1}^{i}$, then $s_{3}=(a_{1}-c_{1})(a_{1}-b_{1})^{2}\ne0$. Hence, there are at most 3 solutions for $z_{1}$. If $y_{1}$ and $z_{1}$ are given, then all the remaining variables are uniquely determined. Hence there are at most $24\ (=8\times3)$ different sequences of $(x,y,z,w)$ for this case.
\end{proof}

Now we regard $f_{i}$ ($i=1,2,\dots,6$) in Eqs. (\ref{subdeq1})-(\ref{subdeq6}) as polynomials with variables $x_{1},y_{1},z_{1},w_{1},\\w_{2},w_{3}$. By the MAGMA Program 8 in the Appendix A, we can get that
\begin{align*}
&R(f_{1},f_{2};x_{1})=g_{1}\cdot(a_{1}-w_{1}),\\
&R(f_{3},f_{4};y_{1})=g_{2}\cdot(b_{1}-w_{1}),\\
&R(f_{5},f_{6};z_{1})=g_{3}\cdot(c_{1}-w_{1}).
\end{align*}
By Lemma~\ref{subdlemma1}, we have $a_{1}\ne w_{1}$ and $c_{1}\ne w_{1}$. Then we can compute to get that
\begin{align*}
&R(g_{1},g_{2};w_{2})=g_{4},\\
&R(g_{1},g_{3};w_{2})=g_{5},\\
&R(g_{4},g_{5};w_{3})=h\cdot (a_{1}-w_{1})^{2}(a_{1}+w_{1})^{4},
\end{align*}
where $h$ is a polynomial of $w_{1}$ with degree $10$. By Lemma~\ref{subdlemma3}, $a_{1}+w_{1}\ne0$, then $h=0$. We can write $h$ as $h=\sum_{i=0}^{10}h_{i}w_{1}^{i}$.

 If at least one of $h_{i}$ is not 0, then there are at most 10 solutions for $w_{1}$. For any fixed $w_{1}$, $g_{4}$ and $g_{5}$ are polynomials of $w_{3}$ with degree 2. We write $g_{4}$ and $g_{5}$ as $g_{4}=\sum_{i=0}^{2}s_{i}w_{3}^{i}$ and $g_{5}=\sum_{i=0}^{2}t_{i}w_{3}^{i}$, then $s_{2}=s_{2}'\cdot(a_{1}-b_{1})$ and $t_{2}=t_{2}'\cdot(a_{1}-c_{1})$. We can compute to get that $s_{2}'-t_{2}'=(b_{1}-c_{1})(a_{1}+b_{1}+c_{1}+w_{1})$. By Lemmas~\ref{subdlemma3} and \ref{subdlemma1}, we have $s_{2}'-t_{2}'\ne0$. Hence there is at least one of $s_{2},t_{2}$ not being 0, then there are at most $2$ solutions for $w_{3}$. If $w_{1}$ and $w_{3}$ are given, then all the remaining variables are uniquely determined. Hence there are at most $20\ (=10\times2)$ different sequences of $(x,y,z,w)$ for this case.

If $h_{i}=0$ for $0\le i\le10$. We regard $h_{i}$ as polynomials with variables $b_{2},c_{2},c_{3}$.  We can compute to get that
\begin{align*}
h_{10}=(a_{1}b_{2}-a_{1}c_{2}-a_{2}b_{1}+a_{2}c_{1}+b_{1}c_{2}-b_{2}c_{1})^2.
\end{align*}
Let
$h_{10}'=a_{1}b_{2}-a_{1}c_{2}-a_{2}b_{1}+a_{2}c_{1}+b_{1}c_{2}-b_{2}c_{1}.$ By the MAGMA Program 8 in the Appendix A, we can get that
\begin{align*}
&R(h_{10}',h_{8};c_{2})=(a_{2}-b_{2})(a_{1}-b_{1})\cdot k_{1},\\
&R(h_{10}',h_{6};c_{2})=(a_{2}-b_{2})(a_{1}-b_{1})\cdot k_{2},\\
&R(k_{1},k_{2};c_{3})=(b_{1}-c_{1})^{4}(a_{1}-c_{1})^{4}(a_{1}-b_{1})^{4}\cdot r_{1}^{2},\\
&R(k_{1},k_{2};b_{2})=(b_{1}-c_{1})^{2}(a_{1}-c_{1})^{2}(a_{1}-b_{1})\cdot r_{2}^{2},
\end{align*}
where
\begin{align}
&r_{1}=a_{1}^{3}a_{2}-a_{1}^{3}b_{2}+a_{1}^{2}a_{2}b_{1}-a_{1}^{2}b_{1}b_{2}-a_{1}a_{2}b_{1}^{2}+a_{1}b_{1}^{2}b_{2}-a_{2}b_{1}^{3}-a_{3}^{2}+2a_{3}b_{3}+b_{1}^{3}b_{2}-b_{3}^{2},\label{subdeq13}\\
&r_{2}=a_{1}^{2}b_{3}-a_{1}^{2}c_{3}-a_{3}b_{1}^{2}+a_{3}c_{1}^{2}+b_{1}^{2}c_{3}-b_{3}c_{1}^{2}.\label{subdeq14}
\end{align}

We can also compute to get that
\begin{align*}
&R(h_{10}',h_{8};b_{2})=(a_{2}-c_{2})(a_{1}-c_{1})\cdot k_{3},\\
&R(h_{10}',h_{6};b_{2})=(a_{2}-c_{2})(a_{1}-c_{1})\cdot k_{4},\\
&R(k_{3},k_{4};c_{3})=(b_{1}-c_{1})^{4}(a_{1}-c_{1})^{2}(a_{1}-b_{1})^{4}\cdot r_{3}^{2},
\end{align*}
where
\begin{align}
r_{3}=a_{1}^{4}a_{2}-a_{1}^{4}c_{2}-2a_{1}^{2}a_{2}b_{1}^{2}+2a_{1}^{2}b_{1}^{2}c_{2}-a_{1}a_{3}^{2}+2a_{1}a_{3}b_{3}-a_{1}b_{3}^{2}+a_{2}b_{1}^{4}+a_{3}^{2}c_{1}-
2a_{3}b_{3}c_{1}-b_{1}^{4}c_{2}+b_{3}^{2}c_{1}.\label{subdeq15}
\end{align}

Now we define $d_{1}:=\frac{a_{3}-b_{3}}{a_{1}^{2}-b_{1}^{2}}$, $d_{2}:=a_{1}^{2}(\frac{a_{3}-b_{3}}{a_{1}^{2}-b_{1}^{2}})-a_{3}$, and $d_{3}:=a_{1}(\frac{a_{3}-b_{3}}{a_{1}^{2}-b_{1}^{2}})^{2}-a_{2}$. Then by $r_{1}=r_{2}=r_{3}=0$ (see Eqs. (\ref{subdeq13}), (\ref{subdeq14}) and (\ref{subdeq15})), it can be computed to get that
\begin{align*}
&a_{2}+d_{3}=a_{1}d_{1}^{2}, && a_{3}+d_{2}=a_{1}^{2}d_{1},\\
&b_{2}+d_{3}=b_{1}d_{1}^{2}, && b_{3}+d_{2}=b_{1}^{2}d_{1},\\
&c_{2}+d_{3}=c_{1}d_{1}^{2}, && c_{3}+d_{2}=c_{1}^{2}d_{1}.
\end{align*}

Then the vertex $d=(d_{1},d_{2},d_{3})$ is a common neighbour of $a,b,c$. Hence the vertices $d,a,b,c,x,\\y,z,w$ form a $\theta_{3,3}$, by Lemma~\ref{subdlemma2} and Claim \ref{subdclaim6}, there are at most $24$ different sequences of $(x,y,z,w)$ for this case.

\begin{remark}
Note that $d_{1}$ may not be in the set $S$, and then $d\not\in V$, but by Remark~\ref{subdrmk1}, we still have the same result.
\end{remark}

Therefore, if $a_{1}\ne b_{1}$, $a_{1}\ne c_{1}$, $b_{1}\ne c_{1}$, $a_{2}\ne b_{2}$, $a_{2}\ne c_{2}$ and $b_{2}\ne c_{2}$, then there are at most $24$ different sequences of $(x,y,z,w)$.
\subsection{Proof of Theorem~\ref{subdthm}}
From the previous discussions, for any given $a,b,c\in V$, there are at most $24$ different sequences of $(x,y,z,w)$ such that $ax,xw,by,yw,cz,zw$ are edges in $G_{p}$. Hence $G_{p}$ is $K_{3,25}'$-free.

\section{Constructions of $\Theta_{3,217}^{B}$-free linear hypergraphs}\label{sectiontheta}
In this section, we describe an algebraic construction of $3$-uniform linear hypergraph with no copy of $\Theta_{3,217}^{B}$ having $n$ vertices and $\Omega(n^{\frac{4}{3}})$ edges.

Let $p$ be a sufficiently large prime number, $\mathbb{F}_{p}$ be the finite field of order $p$. Let
\begin{align*}
&T_{1}=\{2,3,\dots,\frac{p-1}{2}\},\\
&T_{2}=\{\frac{p+3}{2},\dots,p-1\},\\
&T_{3}=\mathbb{F}_{p}\backslash\{-x^{2}:x\in\mathbb{F}_{p}\},\\
&T_{4}=\{x: x\in \mathbb{F}_{p}, x^{2}-4x+1=0\}\cup\{x: x\in \mathbb{F}_{p}, 3x-1=0\}\cup\{x: x\in \mathbb{F}_{p}, 3x-2=0\},\\
&T_{5}=\bigg\{x: x\in \mathbb{F}_{p}, x^5 - \frac{12757}{10872} x^4 + \frac{1123}{3624}x^3 + \frac{289}{1359}x^2 - \frac{49}{453}x -\frac{2}{151}=0\bigg\}.
 \end{align*}
Since $p$ is a sufficiently large prime, then the above definition of $T_{i}$ is well-defined.
 Note that $T_{1}\cup T_{2}\cup\{0,1,\frac{p+1}{2}\}=\mathbb{F}_{p}$ and $|T_{3}|=\frac{p-1}{2}$. Then there exists $i\in\{1,2\}$ such that $|T_{i}\cap T_{3}|\ge\frac{p-7}{4}$. Without loss of generality, we assume that $|T_{1}\cap T_{3}|\ge\frac{p-7}{4}$. Let $S_{1}=(T_{1}\cap T_{3})\backslash (T_{4}\cup T_{5})$, $S_{2}=\mathbb{F}_{p}\backslash\{0,1\}$. Since $|T_{4}|\le4$ and $|T_{5}|\le5$, then $|S_{1}|\ge\frac{p-43}{4}$.

Now we are ready to construct a $3$-partite $3$-uniform linear hypergraph as follows.
 For $1\le i\le 3$, let $V_{i}=S_{1}\times S_{2}\times S_{2}\times \{i\}$. The union $V_{1}\cup V_{2}\cup V_{3}$ will be the vertex set of our hypergraph. Given $x_{1},x_{2},x_{3}\in S_{1}$ and $a\in\mathbb{F}_{p}^{*}$, let
\begin{align*}
e(x_{1},x_{2},x_{3},a)=\{(x_{1},x_{2}x_{3}+a,x_{2}^{2}x_{3}+a,1),(x_{2},x_{3}x_{1}+a,x_{3}^{2}x_{1}+a,2),(x_{3},x_{1}x_{2}+a,x_{1}^{2}x_{2}+a,3)\}.
\end{align*}
\begin{definition}
We define $\mathcal{H}$ to be the $3$-uniform hypergraph with vertex set
\begin{align*}
V(\mathcal{H})=\{(b,c,d,i): b\in S_{1},c,d\in S_{2},1\le i\le3\}
\end{align*}
and edge set
\begin{align*}
E(\mathcal{H})=\{e(x_{1},x_{2},x_{3},a): e(x_{1},x_{2},x_{3},a)\subseteq V(\mathcal{H})\}.
\end{align*}
\end{definition}

\begin{proposition}
$\mathcal{H}$ is a linear hypergraph.
\end{proposition}
\begin{proof}
Let $x_{1},x_{2},x_{3},y_{1},y_{2},y_{3}\in S_{1}$ and $a,b\in\mathbb{F}_{p}^{*}$.
Assume $|e(x_{1},x_{2},x_{3},a)\cap e(y_{1},y_{2},y_{3},b)|=2$, without loss of generality, suppose that the first two vertices are the same. Then we have
\begin{align}
&x_{1}=y_{1},\label{tlheq1}\\
&x_{2}x_{3}+a=y_{2}y_{3}+b,\label{tlheq2}\\
&x_{2}^{2}x_{3}+a=y_{2}^{2}y_{3}+b,\label{tlheq3}\\
&x_{2}=y_{2},\label{tlheq4}\\
&x_{3}x_{1}+a=y_{3}y_{1}+b,\label{tlheq5}\\
&x_{3}^{2}x_{1}+a=y_{3}^{2}y_{1}+b.\label{tlheq6}
\end{align}

Substituting $x_{2}=y_{2}$ (Eq. (\ref{tlheq4})) into Eqs. (\ref{tlheq2}) and (\ref{tlheq3}), then computing Eq. (\ref{tlheq2})-Eq. (\ref{tlheq3}), we have
\[(x_{2}-x_{2}^2)x_{3}=(x_{2}-x_{2}^2)y_{3}.\]
Note that $x_{2}\neq0,1$, we get that $x_{3}=y_{3}$, and then $a=b$.
Hence $e(x_{1},x_{2},x_{3},a)=e(y_{1},y_{2},y_{3},b)$. Therefore $\mathcal{H}$ is linear.
\end{proof}

It is easy to see that the number of vertices of $\mathcal{H}$ is $n:=3|S_{1}|(p-2)^{2}$, and there are at least $|S_{1}|^{3}(p-13)=\Omega(n^{\frac{4}{3}})$ edges in $\mathcal{H}$. In the following of this section, we will prove that $\mathcal{H}$ is $\Theta_{3,217}^{B}$-free.

We call a \emph{Berge $3$-path} $v_{1},e_{1},v_{2},e_{2},v_{3},e_{3},v_{4}$ of type $(i_{1},i_{2},i_{3},i_{4})$ if $v_{j}\in V_{i_{j}}$ for $1\le j\le 4$, and a \emph{Berge $4$-cycle} $v_{1},e_{1},v_{2},e_{2},v_{3},e_{3},v_{4} ,e_{4},v_{1}$ of type $(i_{1},i_{2},i_{3},i_{4})$ if $v_{j}\in V_{i_{j}}$ for $1\le j\le 4$.

By symmetry, without loss of generality, we only need to consider three types of Berge $3$-paths: $(1,2,1,2)$-type, $(1,2,3,1)$-type and $(1,2,3,2)$-type. Indeed, we will upper bound the number of Berge 3-paths of the above three types for any given endpoints.
In the following, we divide our discussions into three subsections according to the types.

\subsection{$(1,2,1,2)$-type Berge $3$-paths}
We begin with the following lemma.
\begin{lemma}\label{thetalemma1}
There is no $(1,2,1,2)$-type Berge $4$-cycle in $\mathcal{H}$.
\end{lemma}
\begin{proof}
Assume to the contrary, suppose $(u_{1},v_{1},w_{1},1),e(x_{1},x_{2},x_{3},a_{1}),(u_{2},v_{2},w_{2},2),e(y_{1},y_{2},y_{3},a_{2}),\\(u_{3},v_{3},w_{3},1),e(z_{1},z_{2},z_{3},a_{3}),(u_{4},v_{4},w_{4},2),e(t_{1},t_{2},t_{3},a_{4}),(u_{1},v_{1},w_{1},1)$ form a copy of Berge $4$-cycle. Then by the definition of hypergraph $\mathcal{H}$, we have
\begin{align}
&u_{1}=x_{1}=t_{1},\label{thetaeq1}\\
&v_{1}=x_{2}x_{3}+a_{1}=t_{2}t_{3}+a_{4},\label{thetaeq2}\\
&w_{1}=x_{2}^{2}x_{3}+a_{1}=t_{2}^{2}t_{3}+a_{4},\label{thetaeq3}\\
&u_{2}=x_{2}=y_{2},\label{thetaeq4}\\
&v_{2}=x_{3}x_{1}+a_{1}=y_{3}y_{1}+a_{2},\label{thetaeq5}\\
&w_{2}=x_{3}^{2}x_{1}+a_{1}=y_{3}^{2}y_{1}+a_{2},\label{thetaeq6}\\
&u_{3}=y_{1}=z_{1},\label{thetaeq7}\\
&v_{3}=y_{2}y_{3}+a_{2}=z_{2}z_{3}+a_{3},\label{thetaeq8}\\
&w_{3}=y_{2}^{2}y_{3}+a_{2}=z_{2}^{2}z_{3}+a_{3},\label{thetaeq9}\\
&u_{4}=z_{2}=t_{2},\label{thetaeq10}\\
&v_{4}=z_{3}z_{1}+a_{3}=t_{3}t_{1}+a_{4},\label{thetaeq11}\\
&w_{4}=z_{3}^{2}z_{1}+a_{3}=t_{3}^{2}t_{1}+a_{4}.\label{thetaeq12}
\end{align}
We can compute to get the following equations.
\begin{align*}
&f_1:=x_2x_3-x_3x_1+t_3x_1-z_2t_3+z_2z_3-z_3y_1+y_3y_1-x_2y_3=0,\\
&f_2:=x_2x_3-x_2^2x_3-z_2t_3+z_2^2t_3=0,\\
&f_3:=x_3x_1-x_3^2x_1-y_3y_1+y_3^2y_1=0,\\
&f_4:=z_3y_1-z_3^2y_1-t_3x_1+t_3^2x_1=0,\\
&f_5:=x_2y_3-x_2^2y_3-z_2z_3+z_2^2z_3=0,
\end{align*}
where $f_{1}$ is from Eqs. (\ref{thetaeq1}), (\ref{thetaeq2}), (\ref{thetaeq4}), (\ref{thetaeq5}), (\ref{thetaeq7}), (\ref{thetaeq8}), (\ref{thetaeq10}) and (\ref{thetaeq11}), $f_{2}$ is from Eqs. (\ref{thetaeq2}), (\ref{thetaeq3}) and (\ref{thetaeq10}), $f_{3}$ is from Eqs. (\ref{thetaeq5}) and (\ref{thetaeq6}), $f_{4}$ is from Eqs. (\ref{thetaeq1}), (\ref{thetaeq7}) (\ref{thetaeq11}) and (\ref{thetaeq12}), and $f_{5}$ is from Eqs. (\ref{thetaeq4}), (\ref{thetaeq8}) and (\ref{thetaeq9}).

Regarding $f_{i}$ as polynomials with variables $x_1,x_2,x_3,y_1,y_3,z_2,z_3,t_3$, we can compute to get the following polynomials.
\begin{align*}
 &g_{1}=R(f_{3},f_{4};x_{1}),\\
 &g_{2}=R(f_{2},f_{5};x_{2}),\\
 &h=R(g_{1},g_{2};x_{3}).
 \end{align*}
By the MAGMA Program 9 in the Appendix B, the polynomial $h$ can be factorized as
\begin{align*}
h=t_{3}^{2}z_{3}^{2}z_{2}^{4}y_{3}^{2}y_{1}^{2}(z_{3}-t_{3})^{2}(z_{2}-1)^{4}(y_{3}-z_{3})^{2}.
\end{align*}

Since $h$ is obtained from $f_{i}$, then $h$ must be $0$. Hence we have  $z_{2}=1$ or $z_{3}=t_{3}$ or $y_{3}=z_{3}$.

If $z_{2}=1$, then $u_{4}=1$, which contradicts the fact that $(u_{4},v_{4},w_{4},2)\in V_{2}$ by the definition of $S_{1}$.

If $z_{3}=t_{3}$, then from Eqs. (\ref{thetaeq11}) and (\ref{thetaeq12}), we have $z_{1}(z_{3}-z_{3}^{2})=t_{1}(z_{3}-z_{3}^{2})$. Note that $z_{3}-z_{3}^{2}\ne0$, then $z_{1}=t_{1}$, and so $a_{3}=a_{4}$. This leads to $e(z_{1},z_{2},z_{3},a_{3})=e(t_{1},t_{2},t_{3},a_{4})$, which is a contradiction.

If $y_{3}=z_{3}$, then from Eqs. (\ref{thetaeq8}) and (\ref{thetaeq9}), we have $y_{2}-y_{2}^{2}=z_{2}-z_{2}^{2}\ne0$. Since $x_{2}=y_{2}$ and $z_{2}=t_{2}$, then $x_{2}-x_{2}^{2}=t_{2}-t_{2}^{2}$. From Eqs. (\ref{thetaeq2}) and (\ref{thetaeq3}), we have $x_{3}=t_{3}$. Substituting the equations $x_{3}=t_{3}$ and $y_{3}=z_{3}$ to $f_{1}$, we get that
\begin{align*}
(x_{2}-z_{2})(x_{3}-y_{3})=0.
\end{align*}
If $x_{3}=y_{3}$, then we have $z_{3}=t_{3}$, which is impossible. If $x_{2}=z_{2}$, then $x_{2}=t_{2}$, and so $a_{1}=a_{4}$, this leads to $e(x_{1},x_{2},x_{3},a_{1})=e(t_{1},t_{2},t_{3},a_{4})$, which is a contradiction.
Hence there is no $(1,2,1,2)$-type Berge $4$-cycle in $\mathcal{H}$.
\end{proof}

Now we consider $(1,2,1,2)$-type Berge $3$-path.
For any given $(b_{1},c_{1},d_{1},1)\in V_{1},(b_{2},c_{2},d_{2},2)\in V_{2}$, suppose there exist $(u_{1},v_{1},w_{1},2)\in V_{2},(u_{2},v_{2},w_{2},1)\in V_{1}$ and $e(x_{1},x_{2},x_{3},a_{1}),e(y_{1},y_{2},y_{3},a_{2}),\\e(z_{1},z_{2},z_{3},a_{3})$ such that $(b_{1},c_{1},d_{1},1),e(x_{1},x_{2},x_{3},a_{1}),(u_{1},v_{1},w_{1},2),e(y_{1},y_{2},y_{3},a_{2}),(u_{2},v_{2},w_{2},1),\\e(z_{1},z_{2},z_{3},a_{3}),(b_{2},c_{2},d_{2},2)$ form a copy of Berge $3$-path.

Then by the definition of hypergraph $\mathcal{H}$, we have
\begin{align}
&b_{1}=x_{1},\label{tlheq7}\\
&c_{1}=x_{2}x_{3}+a_{1},\label{tlheq8}\\
&d_{1}=x_{2}^{2}x_{3}+a_{1},\label{tlheq9}\\
&u_{1}=x_{2}=y_{2},\label{tlheq10}\\
&v_{1}=x_{3}x_{1}+a_{1}=y_{3}y_{1}+a_{2},\label{tlheq11}\\
&w_{1}=x_{3}^{2}x_{1}+a_{1}=y_{3}^{2}y_{1}+a_{2},\label{tlheq12}\\
&b_{2}=z_{2},\label{tlheq13}\\
&c_{2}=z_{3}z_{1}+a_{3},\label{tlheq14}\\
&d_{2}=z_{3}^{2}z_{1}+a_{3},\label{tlheq15}\\
&u_{2}=y_{1}=z_{1},\label{tlheq16}\\
&v_{2}=y_{2}y_{3}+a_{2}=z_{2}z_{3}+a_{3},\label{tlheq17}\\
&w_{2}=y_{2}^{2}y_{3}+a_{2}=z_{2}^{2}z_{3}+a_{3}.\label{tlheq18}
\end{align}
Then we have the following equations.
\begin{align*}
&f_1:=x_2x_3+a_1-c_1=0,\\
&f_2:=x_2^2x_3+a_1-d_1=0,\\
&f_3:=z_3y_1+a_3-c_2=0,\\
&f_4:=z_3^2y_1+a_3-d_2=0,\\
&f_5:=x_3b_1+a_1-y_3y_1-a_2=0,\\
&f_6:=x_3^2b_1+a_1-y_3^2y_1-a_2=0,\\
&f_7:=x_2y_3+a_2-b_2z_3-a_3=0,\\
&f_8:=x_2^2y_3+a_2-b_2^2z_3-a_3=0,
\end{align*}
where $f_{1}$ is from Eq. (\ref{tlheq8}), $f_{2}$ is from Eq. (\ref{tlheq9}), $f_{3}$ is from Eqs. (\ref{tlheq14}) and (\ref{tlheq16}), $f_{4}$ is from Eqs. (\ref{tlheq15}) and (\ref{tlheq16}), $f_{5}$ is from Eqs. (\ref{tlheq7}) and (\ref{tlheq11}), $f_{6}$ is from Eqs. (\ref{tlheq7}) and (\ref{tlheq12}), $f_{7}$ is from Eqs. (\ref{tlheq10}), (\ref{tlheq13}) and (\ref{tlheq17}), $f_{8}$ is from Eqs. (\ref{tlheq10}), (\ref{tlheq13}) and (\ref{tlheq18}).

\begin{lemma}\label{thetalemma2}\
\begin{itemize}
  \item [(1)] $b_{i},c_{i},d_{i},u_{i},v_{i},w_{i},x_{j},y_{j},z_{j}\not\in\{0,1\}$ for $1\le i\le2$ and $1\le j\le 3$,
  \item [(2)] $c_{1}\ne d_{1}$, $x_{2}\ne b_{2}$, $c_{2}\ne d_{2}$,
  \item [(3)] $x_{2}+b_{2}\ne1$, $x_{2}^{2}-x_{2}+c_{1}-d_{1}\ne0$, $b_{1}+b_{2}^{2}\ne0$.
\end{itemize}
\end{lemma}
\begin{proof}
All the statements are immediately from the definition of $\mathcal{H}$. We will only prove $x_{2}^{2}-x_{2}+c_{1}-d_{1}\ne0$. If $x_{2}^{2}-x_{2}+c_{1}-d_{1}=0$, then by the equations $f_{1}$ and $f_{2}$, we have $c_{1}-d_{1}=(x_{2}-x_{2}^{2})x_{3}$, then $x_{3}=1$, which is a contradiction.
\end{proof}
\begin{remark}\label{thetarmk1}
When we consider the resultant of two polynomials, we can divide those nonzero factors first.
\end{remark}

Now we regard $f_{i}$ as polynomials with variables $x_2,x_3,y_1,y_3,z_3,a_1,a_2,a_3$. We can compute to get the following polynomials.
 \begin{align*}
 &g_{1}=R(f_{1},f_{2};a_{1}),&&g_{2}=R(f_{5},f_{6};a_{1}),\\
 &g_{3}=R(f_{1},f_{5};a_{1}),&&g_{4}=R(f_{3},f_{4};a_{3}),\\
 &g_{5}=R(f_{7},f_{8};a_{3}),&&g_{6}=R(f_{3},f_{7};a_{3}),\\
 &g_{7}=R(g_{3},g_{6};a_{2}),&&g_{8}=R(g_{1},g_{2};x_{3}),\\
 &g_{9}=R(g_{1},g_{7};x_{3}),& &g_{10}=R(g_{4},g_{5};z_{3}),\\
 &g_{11}=R(g_{5},g_{9};z_{3}),&&g_{12}=R(g_{8},g_{10};y_{1}),\\
 &g_{13}=R(g_{8},g_{11};y_{1}).
 \end{align*}
 By the MAGMA Program 10 in the Appendix B, the polynomials $g_{12}$ and $g_{13}$ can be factorized as
\begin{align*}
&g_{12}=y_{3}x_{2}(x_{2}-1)g_{14},\\
&g_{13}=y_{3}x_{2}(x_{2}-1)g_{15}.
\end{align*}
Finally, we can get that
\begin{align*}
R(g_{14},g_{15};y_{3})=b_{1}x_{2}^{2}(c_{1}-d_{1})(x_{2}-1)^{2}(x_{2}-b_{2})(x_{2}+b_{2}-1)(x_{2}^{2}-x_{2}+c_{1}-d_{1})g_{16},
\end{align*}
where $g_{16}$ is a polynomial of $x_{2}$ with degree $4$.

By Lemma~\ref{thetalemma2}, we have $b_{1}x_{2}^{2}(c_{1}-d_{1})(x_{2}-1)^{2}(x_{2}-b_{2})(x_{2}+b_{2}-1)(x_{2}^{2}-x_{2}+c_{1}-d_{1})\ne0$. Since $g_{16}$ is obtained from $f_{i}$, then $g_{16}=0$.
Write $g_{16}$ as $g_{16}=h_{4}x_{2}^{4}+h_{3}x_{2}^{3}+h_{2}x_{2}^{2}+h_{1}x_{2}+h_{0}$. We claim that $h_{4},h_{3},h_{2},h_{1},h_{0}$ cannot be all $0$. Assume to the contrary, we regard $h_{i}$ as polynomials with variables $b_{1},c_{1},d_{1},b_{2},c_{2},d_{2}$. $h_{0}$ can be factorized as follows.
\begin{align*}h_{0}=&(-b_1c_1^2 + 2b_1c_1d_1 - b_1c_1b_2^2 + b_1c_1b_2 - b_1d_1^2 + b_1d_1b_2^2 - b_1d_1b_2 - b_2^4c_2 + b_2^4d_2 + 2b_2^3c_2 -2b_2^3d_2 -\\
 &b_2^2c_2 + b_2^2d_2)\cdot h_{5}.
\end{align*}

\begin{claim}\label{claim:g17}
    $g_{17}:=-b_1c_1^2 + 2b_1c_1d_1 - b_1c_1b_2^2 + b_1c_1b_2 - b_1d_1^2 + b_1d_1b_2^2 - b_1d_1b_2 - b_2^4c_2 + b_2^4d_2 + 2b_2^3c_2 -2b_2^3d_2 - b_2^2c_2 + b_2^2d_2\ne 0.$
\end{claim}
\begin{proof}[Proof of Claim~\ref{claim:g17}]
If $g_{17}=0$, by the MAGMA Program 11 in the Appendix B, we have
\begin{align*}
R(g_{17},h_{1};d_{2})=&b_{2}^{3}(b_{2}-1)^{3}(c_{1}-d_{1})b_{1}(b_{1}c_{1}-b_{1}d_{1}-c_{1}b_{2}^{2}+d_{1}b_{2}+b_{2}^{2}c_{2}-b_{2}c_{2}),\\
R(g_{17},h_{1};c_{2})=&b_{2}^{4}(b_{2}-1)^{4}(c_{1}-d_{1})b_{1}(b_{1}c_{1}^{2}-2b_{1}c_{1}d_{1}+b_{1}d_{1}^{2}+c_{1}b_{2}^{4}-c_{1}b_{2}^{3}-d_{1}b_{2}^{3}+d_{1}b_{2}^{2}-b_{2}^{4}d_{2}\\
&+2b_{2}^{3}d_{2}-b_{2}^{2}d_{2}).
\end{align*}
Hence
\begin{align*}
&b_{1}c_{1}-b_{1}d_{1}-c_{1}b_{2}^{2}+d_{1}b_{2}+b_{2}^{2}c_{2}-b_{2}c_{2}=0,\\
&b_{1}c_{1}^{2}-2b_{1}c_{1}d_{1}+b_{1}d_{1}^{2}+c_{1}b_{2}^{4}-c_{1}b_{2}^{3}-d_{1}b_{2}^{3}+d_{1}b_{2}^{2}-b_{2}^{4}d_{2}+2b_{2}^{3}d_{2}-b_{2}^{2}d_{2}=0.
\end{align*}
Let $t_{1}=b_{1},t_{2}=b_{2},t_{3}=\frac{c_{1}-d_{1}}{b_{2}-b_{2}^{2}},a_{4}=\frac{d_{1}-b_{2}c_{1}}{1-b_{2}}$. Then it is easy to get that $(b_{1},c_{1},d_{1},1),(b_{2},c_{2},d_{2},2)\in e(t_{1},t_{2},t_{3},a_{4})$. Hence there exists a $(1,2,1,2)$-type Berge $4$-cycle in $\mathcal{H}$, which contradicts Lemma~\ref{thetalemma1}. This completes the proof.
\end{proof}

Now we can compute to get that
\begin{align*}
R(h_{1},h_{5};c_{1})=b_{2}^{6}b_{1}^{3}(c_{2}-d_{2})^{2}(b_{2}-1)^{6}(b_{1}+b_{2}^{2})^{2}.
\end{align*}
By Lemma~\ref{thetalemma2}, we have $R(h_{1},h_{5};c_{1})\ne0$, which is a contradiction. Hence $h_{4},h_{3},h_{2},h_{1},h_{0}$ cannot be all $0$. Note that $g_{16}=h_{4}x_{2}^{4}+h_{3}x_{2}^{3}+h_{2}x_{2}^{2}+h_{1}x_{2}+h_{0}$, then there are at most $4$ solutions for $x_{2}$.

Now for any fixed $x_{2}$, we consider the polynomial $g_{15}$, which is a polynomial of $y_{3}$ with degree $2$. We write $g_{15}$ as $g_{15}=\sum_{i=0}^{2}l_{i}y_{3}^{i}$, where $l_{2}=b_{2}x_{2}^{3}(x_{2}-1)(x_{2}-b_{2})\ne0$. Hence there are at most $2$ solution for $y_{3}$ when $x_{2}$ is given. If $x_{2}$ and $y_{3}$ are given, then all the remaining variables are uniquely determined.

Hence, for any given $(b_{1},c_{1},d_{1},1)\in V_{1},(b_{2},c_{2},d_{2},2)\in V_{2}$, there are at most $4\times2=8$ Berge $3$-paths of $(1,2,1,2)$-type with $(b_{1},c_{1},d_{1},1),(b_{2},c_{2},d_{2},2)$ being its end core vertices.

\subsection{$(1,2,3,1)$-type Berge $3$-paths}

For any given $(b_{1},c_{1},d_{1},1),(b_{2},c_{2},d_{2},1)\in V_{1}$, suppose there exist $(u_{1},v_{1},w_{1},2)\in V_{2}$ and $(u_{2},v_{2},w_{2},3)\in V_{3}$, and $e(x_{1},x_{2},x_{3},a_{1})$, $e(y_{1},y_{2},y_{3},a_{2})$ and $e(z_{1},z_{2},z_{3},a_{3})$ such that $(b_{1},c_{1},d_{1},1),\\e(x_{1},x_{2},x_{3},a_{1}),(u_{1},v_{1},w_{1},2),e(y_{1},y_{2},y_{3},a_{2}),(u_{2},v_{2},w_{2},3),e(z_{1},z_{2},z_{3},a_{3}),(b_{2},c_{2},d_{2},1)$ form a copy of Berge $3$-path. Then by the definition of hypergraph $\mathcal{H}$, we have
\begin{align}
&b_{1}=x_{1},\label{tlheq19}\\
&c_{1}=x_{2}x_{3}+a_{1},\label{tlheq20}\\
&d_{1}=x_{2}^{2}x_{3}+a_{1},\label{tlheq21}\\
&u_{1}=x_{2}=y_{2},\label{tlheq22}\\
&v_{1}=x_{3}x_{1}+a_{1}=y_{3}y_{1}+a_{2},\label{tlheq23}\\
&w_{1}=x_{3}^{2}x_{1}+a_{1}=y_{3}^{2}y_{1}+a_{2},\label{tlheq24}\\
&b_{2}=z_{1},\label{tlheq25}\\
&c_{2}=z_{2}y_{3}+a_{3},\label{tlheq26}\\
&d_{2}=z_{2}^{2}y_{3}+a_{3},\label{tlheq27}\\
&u_{2}=y_{3}=z_{3},\label{tlheq28}\\
&v_{2}=y_{1}y_{2}+a_{2}=z_{1}z_{2}+a_{3},\label{tlheq29}\\
&w_{2}=y_{1}^{2}y_{2}+a_{2}=z_{1}^{2}z_{2}+a_{3}.\label{tlheq30}
\end{align}
Then we can get the following equations.
\begin{align}
&f_1:=x_2x_3+a_1-c_1=0,\label{thetaequation1}\\
&f_2:=x_2^2x_3+a_1-d_1=0,\\
&f_3:=z_2y_3+a_3-c_2=0,\\
&f_4:=z_2^2y_3+a_3-d_2=0,\\
&f_5:=x_3b_1+a_1-y_3y_1-a_2=0,\\
&f_6:=x_3^2b_1+a_1-y_3^2y_1-a_2=0,\\
&f_7:=y_1x_2+a_2-b_2z_2-a_3=0,\\
&f_8:=y_1^2x_2+a_2-b_2^2z_2-a_3=0,\label{thetaequation2}
\end{align}
where $f_{1}$ is from Eq. (\ref{tlheq20}), $f_{2}$ is from Eq. (\ref{tlheq21}), $f_{3}$ is from Eq. (\ref{tlheq26}), $f_{4}$ is from Eq. (\ref{tlheq27}), $f_{5}$ is from Eqs. (\ref{tlheq19}) and (\ref{tlheq23}), $f_{6}$ is from Eqs. (\ref{tlheq19}) and (\ref{tlheq24}), $f_{7}$ is from Eqs. (\ref{tlheq22}), (\ref{tlheq25}) and (\ref{tlheq29}), $f_{8}$ is from Eqs. (\ref{tlheq22}), (\ref{tlheq25}) and (\ref{tlheq30}).

\begin{lemma}\
\begin{itemize}
  \item [(1)]  $b_{i},c_{i},d_{i},u_{i},v_{i},w_{i},x_{j},y_{j},z_{j}\not\in\{0,1\}$ for $1\le i\le2$ and $1\le j\le 3$,
  \item [(2)] $c_{1}\ne d_{1}$, $c_{2}\ne d_{2}$, $b_{1}+b_{2}\ne1$,
  \item [(3)] $b_{1}^{2}-4b_{1}+1\ne0$, $2b_{2}\ne1$, $3b_{2}\ne1$,
  \item [(4)] $b_2^5 - \frac{12757}{10872}b_2^4 + \frac{1123}{3624}b_2^3 + \frac{289}{1359}b_2^2 - \frac{49}{453}b_2 -\frac{2}{151}\ne0$.
\end{itemize}
\end{lemma}
\begin{proof}
All the statements are immediately from the definition of $\mathcal{H}$.
\end{proof}

Now we regard $f_{i}$ as polynomials with variables $x_2,x_3,y_1,y_3,z_2,a_1,a_2,a_3$. We can get the following polynomials.
\begin{align*}
&g_{1}=R(f_{1},f_{2};a_{1}),&&g_{2}=R(f_{5},f_{6};a_{1}),\\
&g_{3}=R(f_{1},f_{5};a_{1}),&&g_{4}=R(f_{3},f_{4};a_{3}),\\
&g_{5}=R(f_{7},f_{8};a_{3}),&&g_{6}=R(f_{3},f_{7};a_{3}),\\
&g_{7}=R(g_{3},g_{6};a_{2}),&&g_{8}=R(g_{1},g_{2};x_{3}),\\
&g_{9}=R(g_{1},g_{7};x_{3}),&&g_{10}=R(g_{4},g_{5};z_{2}),\\
&g_{11}=R(g_{5},g_{9};z_{2}),&&g_{12}=R(g_{10},g_{11};y_{3}),\\
&g_{13}=R(g_{8},g_{11};y_{3}).
\end{align*}
 By the MAGMA Program 12 in the Appendix B, the polynomials $g_{12}$ and $g_{13}$ can be factorized as
\begin{align*}
&g_{12}=b_2y_1x_2\cdot g_{14},\\
&g_{13}=y_1x_2^2(x_2-1)^2\cdot g_{15}.
\end{align*}
Finally, we can compute to get that
\begin{align}\label{thetaeq13}
g_{16}:=R(g_{14},g_{15};y_{1})=b_2^5(b_2-1)^8x_2^9(x_2-1)\cdot g_{17}^{2}g_{18},
\end{align}
where $g_{17}=x_2^3 + x_2^2c_1 + x_2^2b_2^2 - 2x_2^2b_2 - x_2^2c_2 - x_2^2 - x_2d_1 - x_2b_2^3 +2x_2b_2 + x_2c_2 - b_1c_1 + b_1d_1 + b_2^3 - b_2^2$ and $g_{18}$ is a polynomial of $x_2$ with degree $24$. We write $g_{18}$ as $g_{18}=\sum_{i=0}^{24}h_{i}x_{2}^{i}$. We claim that $h_{i}$ $(0\le i\le24)$ cannot be all $0$, otherwise, we can take 6 coefficients such as
\begin{align*}
&h_0=b_2^2(b_2-1)^2(c_1-d_1)^{10}b_{1}^{5}\cdot k_{0},\\
&h_{1}=b_{2}(b_{2}-1)(c_{1}-d_{1})^{9}b_{1}^{5}\cdot k_{1},\\
&h_{2}=(c_{1}-d_{1})^{8}b_{1}^{4}\cdot k_{2},\\
&h_{3}=(c_{1}-d_{1})^{7}b_{1}^{4}\cdot k_{3},\\
&h_{4}=(c_{1}-d_{1})^{6}b_{1}^{3}\cdot k_{4},\\
&h_{24}=(c_2-d_2)b_2^4(b_2-1)^2\cdot k_{5}.
\end{align*}
Now we regard $k_{i}$ $(i=0,1,2,3,4,5)$ as the polynomials with variables $b_{1},b_{2},c_{1},c_{2},d_{1},d_{2}$. Then we can get the following polynomials.
\begin{align*}
&R(k_{0},k_{1};d_{1})=(c_{2}-d_{2})b_{2}^{2}(b_{2}-1)^{2}\cdot k_{6},\\
&R(k_{0},k_{5};d_{1})=k_{7},\\
&R(k_{0},k_{2};d_{1})=(c_{2}-d_{2})b_{2}^{3}(b_{2}-1)^{3}b_{1}\cdot k_{8},\\
&R(k_{0},k_{3};d_{1})=(c_{2}-d_{2})b_{2}^{2}(b_{2}-1)^{3}\cdot k_{9},\\
&R(k_{0},k_{4};d_{1})=(c_{2}-d_{2})b_{2}^{3}(b_{2}-1)^{3}b_{1}\cdot k_{10},\\
&R(k_{6},k_{7};d_{2})=(c_{1}-c_{2})\cdot k_{11},\\
&R(k_{6},k_{8};d_{2})=(c_{1}-c_{2})b_{1}^{2}(b_{1}-1)^{2}(b_{1}-b_{2})(b_{1}+b_{2}-1)\cdot k_{12},\\
&R(k_{6},k_{9};d_{2})=(c_{1}-c_{2})b_{1}^{4}(b_{1}-1)^{3}(b_{1}-b_{2})(b_{1}+b_{2}-1)\cdot k_{13},\\
&R(k_{6},k_{10};d_{2})=(c_{1}-c_{2})b_{1}^{4}(b_{1}-1)^{4}(b_{1}-b_{2})(b_{1}+b_{2}-1)\cdot k_{14}.
\end{align*}
\begin{claim}\label{claim:b1neqb2}
    $b_{1}\ne b_{2}$ and $c_{1}\ne c_{2}$.
\end{claim}
\begin{proof}[Proof of Claim~\ref{claim:b1neqb2}]
If $b_{1}=b_{2}$, then we substitute this equation into $k_{0}$ to get a polynomial $k_{15}$, and substitute it into $k_{1}$ to get a polynomial $k_{16}$. Then $k_{15}$ and $k_{16}$ can be factorized as
\begin{align*}
&k_{15}=b_{2}(b_{2}-1)\cdot k_{17},\\
&k_{16}=b_{2}(b_{2}-1)\cdot k_{18}.
\end{align*}
We can compute to get that $R(k_{17},k_{18};c_{1})=(c_{2}-d_{2})(b_{2}-\frac{1}{3})(d_{1}-d_{2})$, hence $d_{1}=d_{2}$. Similarly, we can get that $c_{1}=c_{2}$. Therefore $(b_{1},c_{1},d_{1},1)=(b_{2},c_{2},d_{2},1)$, which is a contradiction.

Similarly, we can prove that $c_{1}\ne c_{2}$. This completes the proof of our claim.
\end{proof}
Now we can compute to get that
\begin{align*}
&k_{19}=R(k_{11},k_{12};c_{1}),\\
&R(k_{12},k_{13};c_{1})=b_1^2(b_1-b_2)^2(b_1^2-4b_1+1)^2\cdot k_{20},\\
&R(k_{12},k_{14};c_{1})=b_1^3(b_1-b_2)^3(b_1^2-4b_1+1)^3\cdot k_{21},\\
&k_{22}=R(k_{19},k_{20};b_{1}),\\
&k_{23}=R(k_{19},k_{21};b_{1}),
\end{align*}
where the polynomials $k_{i}\ (i=22,23)$ can be factorized as
\begin{align*}
k_{22}=&b_2^{89}(b_2-1)^{20}\bigg(b_2-\frac{2}{3}\bigg)^{17}\bigg(b_2-\frac{1}{3}\bigg)^{10}\bigg(b_2^5 - \frac{12757}{10872}b_2^4 + \frac{1123}{3624}b_2^3 + \frac{289}{1359}b_2^2 - \frac{49}{453}b_2 -\frac{2}{151}\bigg)^4\cdot k_{24},\\
k_{23}=&b_2^{91}(b_2-1)^{30}\bigg(b_2-\frac{2}{3}\bigg)^{26}\bigg(b_2-\frac{1}{3}\bigg)^{15}\bigg(b_2^5 - \frac{12757}{10872}b_2^4 + \frac{1123}{3624}b_2^3 + \frac{289}{1359}b_2^2 - \frac{49}{453}b_2 -\frac{2}{151}\bigg)^6\cdot k_{25}.
\end{align*}
Finally, we have $R(k_{24},k_{25};b_{2})\ne0$, which is a contradiction. Hence, from Eq. (\ref{thetaeq13}), there are at most 24+3=27 solutions for $x_{2}$.

Now for any fixed $x_{2}$, $g_{15}$ is a polynomial of $y_{1}$ with degree $4$. We write $g_{15}$ as $g_{15}=\sum_{i=0}^{4}l_{i}y_{1}^{i}$, then $l_{4}=b_{2}(b_{2}-1)x_{2}^{4}(x_{2}-1)^{2}\ne0$. Hence there are at most $4$ solutions for $y_{1}$. If $x_{2}$ and $y_{1}$ are given, then all the remaining variables are uniquely determined.

Hence, for any given $(b_{1},c_{1},d_{1},1),(b_{2},c_{2},d_{2},1)\in V_{1}$, there are at most $27\times4=108$ Berge $3$-paths of $(1,2,3,1)$-type with $(b_{1},c_{1},d_{1},1),(b_{2},c_{2},d_{2},1)$ being its end core vertices.

\subsection{$(1,2,3,2)$-type Berge $3$-paths}

For any given $(b_{1},c_{1},d_{1},1)\in V_{1},(b_{2},c_{2},d_{2},2)\in V_{2}$, suppose there exist $(u_{1},v_{1},w_{1},2)\in V_{2},(u_{2},v_{2},w_{2},3)\in V_{3}$ and $e(x_{1},x_{2},x_{3},a_{1}),e(y_{1},y_{2},y_{3},a_{2}),e(z_{1},z_{2},z_{3},a_{3})$ such that $(b_{1},c_{1},d_{1},1)$, $e(x_{1},x_{2},x_{3},a_{1})$, $(u_{1},v_{1},w_{1},2)$, $e(y_{1},y_{2},y_{3},a_{2})$, $(u_{2},v_{2},w_{2},3)$, $e(z_{1},z_{2},z_{3},a_{3})$and $(b_{2},c_{2},d_{2},2)$ form a copy of Berge $3$-path. Then by the definition of hypergraph $\mathcal{H}$, we have
\begin{align}
&b_{1}=x_{1},\label{tlheq31}\\
&c_{1}=x_{2}x_{3}+a_{1},\label{tlheq32}\\
&d_{1}=x_{2}^{2}x_{3}+a_{1},\label{tlheq33}\\
&u_{1}=x_{2}=y_{2},\label{tlheq34}\\
&v_{1}=x_{3}x_{1}+a_{1}=y_{3}y_{1}+a_{2},\label{tlheq35}\\
&w_{1}=x_{3}^{2}x_{1}+a_{1}=y_{3}^{2}y_{1}+a_{2},\label{tlheq36}\\
&b_{2}=z_{2},\label{tlheq37}\\
&c_{2}=z_{3}z_{1}+a_{3},\label{tlheq38}\\
&d_{2}=z_{3}^{2}z_{1}+a_{3},\label{tlheq39}\\
&u_{2}=y_{3}=z_{3},\label{tlheq40}\\
&v_{2}=y_{1}y_{2}+a_{2}=z_{1}z_{2}+a_{3},\label{tlheq41}\\
&w_{2}=y_{1}^{2}y_{2}+a_{2}=z_{1}^{2}z_{2}+a_{3}.\label{tlheq42}
\end{align}
Then we can get the following equations.
\begin{align*}
&f_1:=x_2x_3+a_1-c_1=0,\\
&f_2:=x_2^2x_3+a_1-d_1=0,\\
&f_3:=y_3z_1+a_3-c_2=0,\\
&f_4:=y_3^2z_1+a_3-d_2=0,\\
&f_5:=x_3b_1+a_1-y_3y_1-a_2=0,\\
&f_6:=x_3^2b_1+a_1-y_3^2y_1-a_2=0,\\
&f_7:=y_1x_2+a_2-z_1b_2-a_3=0,\\
&f_8:=y_1^2x_2+a_2-z_1^2b_2-a_3=0,
\end{align*}
where $f_{1}$ is from Eq. (\ref{tlheq32}), $f_{2}$ is from Eq. (\ref{tlheq33}), $f_{3}$ is from Eqs. (\ref{tlheq38}) and (\ref{tlheq40}), $f_{4}$ is from Eqs. (\ref{tlheq39}) and (\ref{tlheq40}), $f_{5}$ is from Eqs. (\ref{tlheq31}) and (\ref{tlheq35}), $f_{6}$ is from Eqs. (\ref{tlheq31}) and (\ref{tlheq36}),  $f_{7}$ is from Eqs. (\ref{tlheq34}), (\ref{tlheq37}) and (\ref{tlheq41}), $f_{8}$ is from Eqs. (\ref{tlheq34}), (\ref{tlheq37}) and (\ref{tlheq42}).

\begin{lemma}\
\begin{itemize}
  \item [(1)]  $b_{i},c_{i},d_{i},u_{i},v_{i},w_{i},x_{j},y_{j},z_{j}\not\in\{0,1\}$ for $1\le i\le2$ and $1\le j\le 3$,
  \item [(2)] $c_{1}\ne d_{1}$, $c_{2}\ne d_{2}$, $x_{2}\ne b_{2}$,
  \item [(3)] $x_{2}^{2}-x_{2}+c_{1}-d_{1}\ne0$.
\end{itemize}
\end{lemma}
\begin{proof}
We will only prove $x_{2}\ne b_{2}$, others are immediately from the definition of $\mathcal{H}$. If $x_{2}=b_{2}$, then $y_{2}=z_{2}$. By the equations on $v_{2}$ and $w_{2}$, we have $y_{1}=z_{1}$, and then $a_{2}=a_{3}$. This leads to $e(y_{1},y_{2},y_{3},a_{2})=e(z_{1},z_{2},z_{3},a_{3})$, which is a contradiction.
\end{proof}

Regarding $f_{i}$ as polynomials with variables $x_2,x_3,y_1,y_3,z_1,a_1,a_2,a_3$, we can compute to get that
 \begin{align*}
&g_{1}=R(f_{1},f_{2};a_{1}),&&g_{2}=R(f_{5},f_{6};a_{1}),\\
&g_{3}=R(f_{1},f_{5};a_{1}),&&g_{4}=R(f_{3},f_{4};a_{3}),\\
&g_{5}=R(f_{7},f_{8};a_{3}),&&g_{6}=R(f_{3},f_{7};a_{3}),\\
&g_{7}=R(g_{3},g_{6};a_{2}),&&g_{8}=R(g_{1},g_{2};x_{3}),\\
&g_{9}=R(g_{1},g_{7};x_{3}),&&g_{10}=R(g_{4},g_{5};z_{1}),\\
&g_{11}=R(g_{4},g_{9};z_{1}),&&g_{12}=R(g_{8},g_{10};y_{1}),\\
&g_{13}=R(g_{8},g_{11};y_{1}).
\end{align*}
By the MAGMA Program 13 in the Appendix B, the polynomials $g_{12}$ and $g_{13}$ can be factorized as
\begin{align*}
&g_{12}=y_{3}^{2}(y_{3}-1)^{2}x_{2}\cdot g_{14},\\
&g_{13}=y_{3}(y_{3}-1)x_{2}(x_{2}-1)\cdot g_{15}.
\end{align*}
Finally, we have
\begin{align*}
R(g_{14},g_{15};y_{3})=x_{2}^{2}(x_{2}-1)^{2}\cdot g_{16},
\end{align*}
where $g_{16}$ is a polynomial of $x_2$ with degree $18$. We write $g_{16}$ as $g_{16}=\sum_{i=0}^{18}h_{i}x_{2}^{i}$. We can compute to get that
\begin{align*}
h_{0}=(c_{1}-d_{1})^{10}b_{1}^{5}(b_{1}-1)\ne0.
\end{align*}
Hence, there are at most $18$ solutions for $x_{2}$.

Now for any fixed $x_{2}$, $g_{14}$ is a polynomial of $y_{3}$ with degree $2$. We write $g_{14}$ as $g_{14}=\sum_{i=0}^{2}l_{i}y_{3}^{i}$, then $l_{2}=x_{2}^{2}(x_{2}-1)^{2}l_{2}'$. We can regard $l_{i}$ as polynomials with variables $x_{2},b_{1},b_{2},c_{1},c_{2},d_{1},d_{2}$. Then by the MAGMA Program 14 in the Appendix B, we have
\begin{align*}
R(l_{0},l_{2}';b_{1})=(c_2-d_2)^2b_2(c_1-d_1)^2x_2^2(x_2-1)^4(x_2-b_2)(x_2^2-x_2+c_1-d_1)^2\ne0.
\end{align*}
Hence there are at most $2$ solutions for $y_{3}$. If $x_{2}$ and $y_{3}$ are given, then all the remaining variables are uniquely determined.

Hence, for any given $(b_{1},c_{1},d_{1},1)\in V_{1},(b_{2},c_{2},d_{2},2)\in V_{2}$, there are at most $18\times2=36$ Berge $3$-paths of $(1,2,3,2)$-type with $(b_{1},c_{1},d_{1},1),(b_{2},c_{2},d_{2},2)$ being its end core vertices.

\subsection{Proof of Theorem~\ref{thetathm}}

If the given two vertices are in the same part, without loss of generality, suppose $(b_{1},c_{1},d_{1},1), \\(b_{2},c_{2},d_{2},1)\in V_{1}$. Then there are two types of Berge $3$-paths: $(1,2,3,1)$-type and $(1,3,2,1)$-type. From the previous discussions, there are at most $108\times2=216$ such Berge $3$-paths in $\mathcal{H}$.

If the given two vertices are in different parts, without loss of generality, suppose $(b_{1},c_{1},d_{1},1)\in V_{1}, (b_{2},c_{2},d_{2},2)\in V_{2}$. Then there are three types of Berge $3$-paths: $(1,2,1,2)$-type, $(1,2,3,2)$-type and $(1,3,1,2)$-type. From the previous discussions, there are at most $8+36\times2=80$ such Berge $3$-paths in $\mathcal{H}$. Hence $\mathcal{H}$ is $\Theta_{3,217}^{B}$-free.

\section{Concluding remarks}\label{sectionconclusion}
In this paper, we determine the asymptotics of Tur\'{a}n number of $K_{3,25}'$ and linear Tur\'{a}n number of Berge theta hypergraph $\Theta_{3,217}^{B}$. Our main technique is the polynomial resultant. The parameters 25 and 217 may not be the best possible, hence improving these parameters will be interesting. As we have mentioned in Section~\ref{sectiotech}, polynomial resultant can settle some cases which are difficult to deal with by manual computation. We believe that the polynomial resultant can be used to construct more extremal graphs, it would be interesting to find more such examples.

\section*{Acknowledgements}
The authors thank Prof. Boris Bukh for his helpful discussions and kind suggestions. The research of Zixiang Xu was supported by IBS-R029-C4. The research of G. Ge was supported by the National Key Research and Development Program of China under Grant 2020YFA0712100, the National Natural Science Foundation of China under  Grant 12231014, and Beijing Scholars Program.

\section*{Appendix A}\label{appendix}
\begin{program}{MAGMA program for Subsection~\ref{subsec1} (Case $a_{1}\ne c_{1}$):}\label{subdprogram1}
\begin{align*}
&P<a1,a2,a3,b1,b2,b3,c1,c2,c3,w1,w2,w3,x1,y1,z1>:=PolynomialRing(RationalField(),15);\\
&f1:=a2-w2-(a1-w1)*x1^2;\\
&f2:=a3-w3-(a1^2-w1^2)*x1;\\
&f3:=b2-w2-(a1-w1)*y1^2;\\
&f4:=b3-w3-(a1^2-w1^2)*y1;\\
&f5:=c2-w2-(c1-w1)*z1^2;\\
&f6:=c3-w3-(c1^2-w1^2)*z1;\\
&g1:=Resultant(f1,f2,x1)\text{ div }(a1-w1);\\
&g2:=Resultant(f3,f4,y1)\text{ div }(a1-w1);\\
&g3:=Resultant(f5,f6,z1)\text{ div }(c1-w1);\\
&g4:=Resultant(g1,g2,w2)\text{ div }((a1-w1)*(a1+w1)^2);\\
&g5:=Resultant(g1,g3,w2);\\
&h:=Resultant(g4,g5,w3);\\
&h8:=Coefficient(h,w1,8);\\
&s1:=Coefficient(g4,w3,1);\\
&Factorization(h8);
\end{align*}
\end{program}

\begin{program}{MAGMA program for Subsection~\ref{subsec1} (Case $a_{1}=c_{1}$):}\label{subdprogram2}
\begin{align*}
&P<a1,a2,a3,b1,b2,b3,c1,c2,c3,w1,w2,w3,x1,y1,z1>:=PolynomialRing(RationalField(),15);\\
&f1:=a2-w2-(a1-w1)*x1^2;\\
&f2:=a3-w3-(a1^2-w1^2)*x1;\\
&f3:=b2-w2-(a1-w1)*y1^2;\\
&f4:=b3-w3-(a1^2-w1^2)*y1;\\
&f5:=c2-w2-(a1-w1)*z1^2;\\
&f6:=c3-w3-(a1^2-w1^2)*z1;\\
&g1:=Resultant(f1,f2,x1)\text{ div }(a1-w1);\\
&g2:=Resultant(f3,f4,y1)\text{ div }(a1-w1);\\
&g3:=Resultant(f5,f6,z1)\text{ div } (a1-w1);\\
&g4:=Resultant(g1,g2,w2)\text{ div } ((a1-w1)*(a1+w1)^2);\\
&g6:=Resultant(g1,g3,w2)\text{ div } ((a1-w1)*(a1+w1)^2);\\
&h':=Resultant(g4,g6,w3);\\
&k3:=Coefficient(h',w1,3);\\
&k0:=Coefficient(h',w1,0);\\
&k:=Resultant(k3,k0,a2);\\
&Factorization(k);
\end{align*}
\end{program}

\begin{program}{MAGMA program for Subsection~\ref{subsec2} (Case $a_{2}\ne c_{2}$):}\label{subdprogram3}
\begin{align*}
&P<a1,a2,a3,b1,b2,b3,c1,c2,c3,w1,w2,w3,x1,y1,z1>:=PolynomialRing(RationalField(),15);\\
&f1:=a2-w2-(a1-w1)*x1^2;\\
&f2:=a3-w3-(a1^2-w1^2)*x1;\\
&f3:=a2-w2-(b1-w1)*y1^2;\\
&f4:=b3-w3-(b1^2-w1^2)*y1;\\
&f5:=c2-w2-(c1-w1)*z1^2;\\
&f6:=c3-w3-(c1^2-w1^2)*z1;\\
&g1:=Resultant(f1,f2,x1) \text{ div } (a1-w1);\\
&g2:=Resultant(f3,f4,y1) \text{ div } (b1-w1);\\
&g3:=Resultant(f5,f6,z1) \text{ div } (c1-w1);\\
&g4:=Resultant(g1,g2,w2);\\
&g5:=Resultant(g1,g3,w2);\\
&h:=Resultant(g4,g5,w3) \text{ div } ((a1-w1)^2*(a1+w1)^4);\\
&h10:=Coefficient(h,w1,10);\\
&s2:=Coefficient(g4,w3,2) \text{ div } (a1-b1);\\
&t2:=Coefficient(g5,w3,2) \text{ div } (a1-c1);\\
&Factorization(h10);\\
&Factorization(s2-t2);
\end{align*}
\end{program}

\begin{program}{MAGMA program for Subsection~\ref{subsec2} (Case $a_{2}=c_{2}$ and $(a_{3}-b_{3})(b_{3}-c_{3})(a_{3}-c_{3})\ne0$):}\label{subdprogram4}
\begin{align*}
&P<a1,a2,a3,b1,b2,b3,c1,c2,c3,w1,w2,w3,x1,y1,z1>:=PolynomialRing(RationalField(),15);\\
&f1:=a2-w2-(a1-w1)*x1^2;\\
&f2:=a3-w3-(a1^2-w1^2)*x1;\\
&f3:=a2-w2-(b1-w1)*y1^2;\\
&f4:=b3-w3-(b1^2-w1^2)*y1;\\
&f5:=a2-w2-(c1-w1)*z1^2;\\
&f6:=c3-w3-(c1^2-w1^2)*z1;\\
&g1:=Resultant(f1,f2,x1) \text{ div } (a1-w1);\\
&g2:=Resultant(f3,f4,y1) \text{ div } (b1-w1);\\
&g3:=Resultant(f5,f6,z1) \text{ div } (c1-w1);\\
&g4:=Resultant(g1,g2,w2);\\
&g5:=Resultant(g1,g3,w2);\\
&h:=Resultant(g4,g5,w3) \text{ div } ((a1-w1)^2*(a1+w1)^4);\\
&h5:=Coefficient(h,w1,5);\\
&h4:=Coefficient(h,w1,4);\\
&Factorization(Resultant(h4,h5,a1));
\end{align*}
\end{program}

\begin{program}{MAGMA program for Subsection~\ref{subsec2} (Case $a_{2}=c_{2}$ and $a_{3}=b_{3}$):}\label{subdprogram5}
\begin{align*}
&P<a1,a2,a3,b1,b2,b3,c1,c2,c3,w1,w2,w3,x1,y1,z1>:=PolynomialRing(RationalField(),15);\\
&f1:=a2-w2-(a1-w1)*x1^2;\\
&f2:=a3-w3-(a1^2-w1^2)*x1;\\
&f3:=a2-w2-(b1-w1)*y1^2;\\
&f4:=a3-w3-(b1^2-w1^2)*y1;\\
&f5:=a2-w2-(c1-w1)*z1^2;\\
&f6:=c3-w3-(c1^2-w1^2)*z1;\\
&g1:=Resultant(f1,f2,x1) \text{ div } (a1-w1);\\
&g2:=Resultant(f3,f4,y1) \text{ div } (b1-w1);\\
&g3:=Resultant(f5,f6,z1) \text{ div } (c1-w1);\\
&g4:=Resultant(g1,g2,w2) \text{ div }((a3-w3)^2*(a1-b1));\\
&g5:=Resultant(g1,g3,w2);\\
&t2:=Coefficient(g5,w3,2) \text{ div } (a1-c1);\\
&Factorization(g4-t2);
\end{align*}
\end{program}

\begin{program}{MAGMA program for Lemma~\ref{subdlemma2}:}\label{subdprogram6}
\begin{align*}
&P<a,b,c,d,x,y,z,w>:=PolynomialRing(RationalField(),8);\\
&g1:=(y+z)*(b+c)-(w-a)*(d-x);\\
&g2:=(x+z)*(a+c)-(w-b)*(d-y);\\
&g3:=(x+y)*(a+b)-(w-c)*(d-z);\\
&h1:=Resultant(g1,g2,d) \text{ div } (a+b+c-w);\\
&h2:=Resultant(g1,g3,d) \text{ div } (a+b+c-w);\\
&Factorization(Resultant(h1,h2,w));
\end{align*}
\end{program}

\begin{program}{MAGMA program for Claim~\ref{subdclaim6}:}\label{subdprogram8}
\begin{align*}
&P<a1,a2,a3,b1,b2,b3,c1,c2,c3,w1,w2,w3,x1,y1,z1>:=PolynomialRing(RationalField(),15);\\
&f1:=a2-w2-(a1-w1)*x1^2;\\
&f2:=a3-w3-(a1^2-w1^2)*x1;\\
&f3:=b2-w2-(b1-w1)*y1^2;\\
&f4:=b3-w3-(b1^2-w1^2)*y1;\\
&f5:=c2-w2-(c1-w1)*z1^2;\\
&f6:=c3-w3-(c1^2-w1^2)*z1;\\
&f7:=(a1-b1)*(y1-z1)-(b1-c1)*(x1-y1);\\
&g1:=f1-f3;\\
&g2:=f1-f5;\\
&g3:=f2-f4;\\
&g5:=Resultant(g1,g2,w1);\\
&g6:=Resultant(g1,g3,w1) \text{ div } (x1-y1);\\
&h1:=Resultant(f7,g5,x1) \text{ div } (y1-z1);\\
&h2:=Resultant(f7,g6,x1);\\
&s:=Resultant(h1,h2,z1) \text{ div }((b1-c1)^3*(a1-b1)^6*(a1*y1^2-a2-b1*y1^2+b2)^2);\\
&s8:=Coefficient(s,y1,8) \text{ div } ((b1-c1)^3*(a1-c1)^4*(a1-b1)^4);\\
&s7:=Coefficient(s,y1,7) \text{ div } ((b1-c1)^3*(a1-c1)^3*(a1-b1)^3*(a3-b3));\\
&Factorization(Resultant(s7,s8,a1));
\end{align*}
\end{program}

\begin{program}{MAGMA program for Subsection~\ref{subsec3}:}\label{subdprogram9}
\begin{align*}
&P<a1,a2,a3,b1,b2,b3,c1,c2,c3,w1,w2,w3,x1,y1,z1>:=PolynomialRing(RationalField(),15);\\
&f1:=a2-w2-(a1-w1)*x1^2;\\
&f2:=a3-w3-(a1^2-w1^2)*x1;\\
&f3:=b2-w2-(b1-w1)*y1^2;\\
&f4:=b3-w3-(b1^2-w1^2)*y1;\\
&f5:=c2-w2-(c1-w1)*z1^2;\\
&f6:=c3-w3-(c1^2-w1^2)*z1;\\
&g1:=Resultant(f1,f2,x1) \text{ div } (a1-w1);\\
&g2:=Resultant(f3,f4,y1) \text{ div } (b1-w1);\\
&g3:=Resultant(f5,f6,z1) \text{ div } (c1-w1);\\
&g4:=Resultant(g1,g2,w2);\\
&g5:=Resultant(g1,g3,w2);\\
&h:=Resultant(g4,g5,w3) \text{ div } ((a1-w1)^2*(a1+w1)^4);\\
&h10:=a1*b2-a1*c2-a2*b1+a2*c1+b1*c2-b2*c1;\\
&h8:=Coefficient(h,w1,8);\\
&h6:=Coefficient(h,w1,6);\\
&k1:=Resultant(h10,h8,c2) \text{ div } ((a2-b2)*(a1-b1));\\
&k2:=Resultant(h10,h6,c2) \text{ div } ((a2-b2)*(a1-b1));\\
&Factorization(Resultant(k1,k2,c3));\\
&Factorization(Resultant(k1,k2,b2));\\
&k3:=Resultant(h10,h8,b2) \text{ div } ((a2-c2)*(a1-c1));\\
&k4:=Resultant(h10,h6,b2) \text{ div } ((a2-c2)*(a1-c1));\\
&Factorization(Resultant(k3,k4,c3));
\end{align*}
\end{program}

\section*{Appendix B}
\begin{program}{MAGMA program for (1,2,1,2)-type: Lemma 3.3}\label{thetaprogram1}
\begin{align*}
&P<x1,x2,x3,y1,y3,z2,z3,t3>:=PolynomialRing(RationalField(),8);\\
&f1:=x2*x3-x3*x1+t3*x1-z2*t3+z2*z3-z3*y1+y3*y1-x2*y3;\\
&f2:=x2*x3-x2^2*x3-z2*t3+z2^2*t3;\\
&f3:=x3*x1-x3^2*x1-y3*y1+y3^2*y1;\\
&f4:=z3*y1-z3^2*y1-t3*x1+t3^2*x1;\\
&f5:=x2*y3-x2^2*y3-z2*z3+z2^2*z3;\\
&g1:=Resultant(f3,f4,x1);\\
&g2:=Resultant(f2,f5,x2);\\
&h:=Resultant(g1,g2,x3);\\
&Factorization(h);
\end{align*}
\end{program}

\begin{program}{MAGMA program for (1,2,1,2)-type: General}\label{thetaprogram2}
\begin{align*}
&P<x2,x3,y1,y3,z3,a1,a2,a3,b1,c1,d1,b2,c2,d2>:=PolynomialRing(RationalField(),14);\\
&f1:=x2*x3+a1-c1;\\
&f2:=x2^2*x3+a1-d1;\\
&f3:=z3*y1+a3-c2;\\
&f4:=z3^2*y1+a3-d2;\\
&f5:=x3*b1+a1-y3*y1-a2;\\
&f6:=x3^2*b1+a1-y3^2*y1-a2;\\
&f7:=x2*y3+a2-b2*z3-a3;\\
&f8:=x2^2*y3+a2-b2^2*z3-a3;\\
&g1:=Resultant(f2,f1,a1);\\
&g2:=Resultant(f6,f5,a1);\\
&g3:=Resultant(f5,f1,a1);\\
&g4:=Resultant(f3,f4,a3);\\
&g5:=Resultant(f8,f7,a3);\\
&g6:=Resultant(f7,f3,a3);\\
&g7:=Resultant(g6,g3,a2);\\
&g8:=Resultant(g2,g1,x3);\\
&g9:=Resultant(g7,g1,x3);\\
&g10:=Resultant(g4,g5,z3);\\
&g11:=Resultant(g9,g5,z3);\\
&g14:=Resultant(g8,g10,y1) div (y3*x2*(x2-1));\\
&g15:=Resultant(g8,g11,y1) div (y3*x2*(x2-1));\\
&g16:=Resultant(g14,g15,y3) div (b1*x2^2*(c1-d1)*(x2-1)^2*(x2-b2)*(x2+b2-1)\\
&*(x2^2-x2+c1-d1));\\
&h4:=Coefficient(g16,x2,4);\\
&h3:=Coefficient(g16,x2,3);\\
&h2:=Coefficient(g16,x2,2);\\
&h1:=Coefficient(g16,x2,1);\\
&h5:=Coefficient(g16,x2,0) div (-b1*c1^2 + 2*b1*c1*d1 - b1*c1*b2^2 + b1*c1*b2 - b1*d1^2 + \\
&b1*d1*b2^2 - b1*d1*b2 - b2^4*c2 + b2^4*d2 + 2*b2^3*c2 - 2*b2^3*d2 - b2^2*c2 + b2^2*d2);\\
&Factorization(Resultant(h1,h5,c1));
\end{align*}
\end{program}

\begin{program}{MAGMA program for (1,2,1,2)-type: Claim}\label{thetaprogram3}
\begin{align*}
&P<x2,x3,y1,y3,z3,a1,a2,a3,b1,c1,d1,b2,c2,d2>:=PolynomialRing(RationalField(),14);\\
&f1:=x2*x3+a1-c1;\\
&f2:=x2^2*x3+a1-d1;\\
&f3:=z3*y1+a3-c2;\\
&f4:=z3^2*y1+a3-d2;\\
&f5:=x3*b1+a1-y3*y1-a2;\\
&f6:=x3^2*b1+a1-y3^2*y1-a2;\\
&f7:=x2*y3+a2-b2*z3-a3;\\
&f8:=x2^2*y3+a2-b2^2*z3-a3;\\
&g1:=Resultant(f2,f1,a1);\\
&g2:=Resultant(f6,f5,a1);\\
&g3:=Resultant(f5,f1,a1);\\
&g4:=Resultant(f3,f4,a3);\\
&g5:=Resultant(f8,f7,a3);\\
&g6:=Resultant(f7,f3,a3);\\
&g7:=Resultant(g6,g3,a2);\\
&g8:=Resultant(g2,g1,x3);\\
&g9:=Resultant(g7,g1,x3);\\
&g10:=Resultant(g4,g5,z3);\\
&g11:=Resultant(g9,g5,z3);\\
&g14:=Resultant(g8,g10,y1) div (y3*x2*(x2-1));\\
&g15:=Resultant(g8,g11,y1) div (y3*x2*(x2-1));\\
&g16:=Resultant(g14,g15,y3) div (b1*x2^2*(c1-d1)*(x2-1)^2*(x2-b2)*(x2+b2-1)\\
&*(x2^2-x2+c1-d1));\\
&h1:=Coefficient(g16,x2,1);\\
&g17:=-b1*c1^2 + 2*b1*c1*d1 - b1*c1*b2^2 + b1*c1*b2 - b1*d1^2 + b1*d1*b2^2 - b1*d1*b2\\
& - b2^4*c2 + b2^4*d2 + 2*b2^3*c2 - 2*b2^3*d2 - b2^2*c2 + b2^2*d2;\\
&Factorization(Resultant(g17,h1,d2));\\
&Factorization(Resultant(g17,h1,c2));
\end{align*}
\end{program}

\begin{program}{MAGMA program for (1,2,3,1)-type}\label{thetaprogram4}
\begin{align*}
&P<x2,x3,y1,y3,z2,a1,a2,a3,b1,c1,d1,b2,c2,d2>:=PolynomialRing(RationalField(),14);\\
&f1:=x2*x3+a1-c1;\\
&f2:=x2^2*x3+a1-d1;\\
&f3:=z2*y3+a3-c2;\\
&f4:=z2^2*y3+a3-d2;\\
&f5:=x3*b1+a1-y3*y1-a2;\\
&f6:=x3^2*b1+a1-y3^2*y1-a2;\\
&f7:=y1*x2+a2-b2*z2-a3;\\
&f8:=y1^2*x2+a2-b2^2*z2-a3;\\
&g1:=Resultant(f2,f1,a1);\\
&g2:=Resultant(f6,f5,a1);\\
&g3:=Resultant(f5,f1,a1);\\
&g4:=Resultant(f3,f4,a3);\\
&g5:=Resultant(f8,f7,a3);\\
&g6:=Resultant(f7,f3,a3);\\
&g7:=Resultant(g6,g3,a2);\\
&g8:=Resultant(g2,g1,x3);\\
&g9:=Resultant(g7,g1,x3);\\
&g10:=Resultant(g5,g4,z2);\\
&g11:=Resultant(g9,g5,z2);\\
&g14:=Resultant(g10,g11,y3) div (b2*y1*x2);\\
&g15:=Resultant(g8,g11,y3) div (y1*x2^2*(x2-1)^2);\\
&g18:=Resultant(g14,g15,y1) div (b2^5*(b2-1)^8*x2^9*(x2-1)*(x2^3 + x2^2*c1 + x2^2*b2^2 - \\
&2*x2^2*b2 - x2^2*c2 - x2^2 - x2*d1 - x2*b2^3 +2*x2*b2 + x2*c2 - b1*c1 + b1*d1 + b2^3 - b2^2)^2);\\
&k0:=Coefficient(g18,x2,0) div (b2^2*(b2-1)^2*(c1-d1)^{10}*b1^5);\\
&k1:=Coefficient(g18,x2,1) div (b2*(b2-1)*(c1-d1)^9*b1^5);\\
&k2:=Coefficient(g18,x2,2) div ((c1-d1)^8*b1^4);\\
&k3:=Coefficient(g18,x2,3) div ((c1-d1)^7*b1^4);\\
&k4:=Coefficient(g18,x2,4) div ((c1-d1)^6*b1^3);\\
&k5:=Coefficient(g18,x2,24) div ((c2-d2)*b2^4*(b2-1)^2);\\
&k6:=Resultant(k0,k1,d1) div ((c2-d2)*b2^2*(b2-1)^2);\\
&k7:=Resultant(k0,k5,d1);
\end{align*}
\end{program}

\begin{align*}
&k8:=Resultant(k0,k2,d1) div ((c2-d2)*b2^3*(b2-1)^3*b1);\\
&k9:=Resultant(k0,k3,d1) div ((c2-d2)*b2^2*(b2-1)^3);\\
&k10:=Resultant(k0,k4,d1) div ((c2-d2)*b2^3*(b2-1)^3*b1);\\
&k11:=Resultant(k6,k7,d2) div (c1-c2);\\
&k12:=Resultant(k6,k8,d2) div ((c1-c2)*b1^2*(b1-1)^2*(b1-b2)*(b1+b2-1));\\
&k13:=Resultant(k6,k9,d2) div ((c1-c2)*b1^4*(b1-1)^3*(b1-b2)*(b1+b2-1));\\
&k14:=Resultant(k6,k10,d2) div ((c1-c2)*b1^4*(b1-1)^4*(b1-b2)*(b1+b2-1));\\
&k19:=Resultant(k11,k12,c1);\\
&k20:=Resultant(k12,k13,c1) div (b1^2*(b1-b2)^2*(b1^2-4*b1+1)^2);\\
&k21:=Resultant(k12,k14,c1) div (b1^3*(b1-b2)^3*(b1^2-4*b1+1)^3);\\
&k24:=Resultant(k19,k20,b1) div (b2^{89}*(b2-1)^{20}*(b2-2/3)^{17}*(b2-1/3)^{10}*\\
&(b2^5-12757/10872*b2^4+1123/3624*b2^3+289/1359*b2^2-49/453*b2-2/151)^4);\\
&k25:=Resultant(k19,k21,b1) div (b2^{91}*(b2-1)^{30}*(b2-2/3)^{26}*(b2-1/3)^{15}*\\
&(b2^5-12757/10872*b2^4+1123/3624*b2^3+289/1359*b2^2-49/453*b2-2/151)^6);\\
&Resultant(k24,k25,b2);
\end{align*}

\begin{program}{MAGMA program for (1,2,3,2)-type: Determine the number of solutions for $x2$}\label{thetaprogram5}
\begin{align*}
&P<x2,x3,y1,y3,z1,a1,a2,a3,b1,c1,d1,b2,c2,d2>:=PolynomialRing(RationalField(),14);\\
&f1:=x2*x3+a1-c1;\\
&f2:=x2^2*x3+a1-d1;\\
&f3:=y3*z1+a3-c2;\\
&f4:=y3^2*z1+a3-d2;\\
&f5:=x3*b1+a1-y3*y1-a2;\\
&f6:=x3^2*b1+a1-y3^2*y1-a2;\\
&f7:=y1*x2+a2-z1*b2-a3;\\
&f8:=y1^2*x2+a2-z1^2*b2-a3;\\
&g1:=Resultant(f2,f1,a1);\\
&g2:=Resultant(f6,f5,a1);\\
&g3:=Resultant(f5,f1,a1);\\
&g4:=Resultant(f3,f4,a3);\\
&g5:=Resultant(f8,f7,a3);\\
&g6:=Resultant(f7,f3,a3);\\
&g7:=Resultant(g6,g3,a2);\\
&g8:=Resultant(g2,g1,x3);\\
&g9:=Resultant(g7,g1,x3);\\
&g10:=Resultant(g5,g4,z1);\\
&g11:=Resultant(g9,g4,z1);\\
&g14:=Resultant(g8,g10,y1) div (y3^2*(y3-1)^2*x2);\\
&g15:=Resultant(g8,g11,y1) div (y3*(y3-1)*x2*(x2-1));\\
&g16:=Resultant(g14,g15,y3) div (x2^2*(x2-1)^2);\\
&Factorization(Coefficient(g16,x2,0));
\end{align*}
\end{program}

\begin{program}{MAGMA program for (1,2,3,2)-type: Given $x2$, determine the number of solutions for $y3$}\label{thetaprogram6}
\begin{align*}
&P<x2,x3,y1,y3,z1,a1,a2,a3,b1,c1,d1,b2,c2,d2>:=PolynomialRing(RationalField(),14);\\
&f1:=x2*x3+a1-c1;\\
&f2:=x2^2*x3+a1-d1;\\
&f3:=y3*z1+a3-c2;\\
&f4:=y3^2*z1+a3-d2;\\
&f5:=x3*b1+a1-y3*y1-a2;\\
&f6:=x3^2*b1+a1-y3^2*y1-a2;\\
&f7:=y1*x2+a2-z1*b2-a3;\\
&f8:=y1^2*x2+a2-z1^2*b2-a3;\\
&g1:=Resultant(f2,f1,a1);\\
&g2:=Resultant(f6,f5,a1);\\
&g4:=Resultant(f3,f4,a3);\\
&g5:=Resultant(f8,f7,a3);\\
&g8:=Resultant(g2,g1,x3);\\
&g10:=Resultant(g5,g4,z1);\\
&g14:=Resultant(g8,g10,y1) div (y3^2*(y3-1)^2*x2);\\
&l2:=Coefficient(g14,y3,2) div (x2^2*(x2-1)^2);\\
&l0:=Coefficient(g14,y3,0);\\
&Factorization(Resultant(l2,l0,b1));
\end{align*}
\end{program}

\end{document}